\documentclass[a4paper,11 pt]{amsart}

\usepackage{eulervm}
\usepackage{hyperref}

\usepackage[utf8]{inputenc}
\usepackage[T1]{fontenc}
\usepackage[english]{babel}
\usepackage{mathtools}
\usepackage{braket, fontenc}
\usepackage{amsmath, amssymb, stmaryrd, mathabx}
\usepackage{enumerate}
\usepackage{enumitem}
\usepackage{tikz-cd}
\usepackage{tikz,float, hyperref, circuitikz, comment}

\newcommand{\rr}{{\Phi}}
\newcommand{\scp}[2]{\langle #1, #2\rangle}

\newcommand{\Ps}{\mathcal{P}}
\newcommand{\Ms}{\mathcal{M}}
\newcommand{\res}{\mathrm{Res}}

\newcommand{\atil}{\widetilde{a}}
\newcommand{\btil}{\widetilde{b}}
\newcommand{\ftil}{\widetilde{f}}
\newcommand{\x}{z}
\newcommand{\xtil}{\widetilde{z}}
\newcommand{\R}{\mathbb{R}}
\newcommand{\Z}{\mathbb{Z}}
\newcommand{\C}{\mathbb{C}}

\newcommand{\HH}{\mathcal{H}}
\newcommand{\bb}{\textbf{B}}
\newcommand{\Bases}{\mathcal{B}}
\newcommand{\iber}{\mathrm{iBer}}
\newcommand{\G}{\mathcal{G}}
\newcommand{\F}{\mathcal{F}}
\newcommand{\xb}{\overline{x}}
\newcommand{\nn}{N}

 \theoremstyle{plain}
 \newtheorem{theorem}{Theorem}[section]
\newtheorem{proposition}[theorem]{Proposition}
 \newtheorem{lemma}[theorem]{Lemma}
 
\newtheorem{corollary}[theorem]{Corollary}
\theoremstyle{definition}
 \newtheorem{definition}[theorem]{Definition}

 \newtheorem{remark}[theorem]{Remark}

\title[Intersection theory on singular moduli spaces of vector bundles]{Intersection theory on singular moduli spaces of vector bundles: a parabolic approach}
\author{Camilla Felisetti}
\address{ Departimento di Scienze Fisiche, Informatiche e Matematiche, UniversitÃ  di Modena e Reggio Emilia.}
\email{Camilla.Felisetti@unimore.it}

\author{Olga Trapeznikova}
\address{ Faculty of Mathematics, ISTA}
\email{Olga.Trapeznikova@ist.ac.at}

\begin{document}	
\begin{abstract}
We present explicit formulas for the intersection pairing in the intersection cohomology of the moduli space $\Ms_0(r)$ of rank-$r$, degree-$0$ semistable bundles on a Riemann surface. The key idea is to realize this intersection cohomology as a canonical subspace of the cohomology of a  smooth moduli space of parabolic bundles, where the pairing can be computed via the Hecke correspondence and the Jeffrey–Kirwan iterated residue formulas. This approach provides a simpler alternative to the blow-up construction of Jeffrey--Kirwan--Kiem--Woolf, yielding formulas for the intersection pairing on $\Ms_0(r)$, for arbitrary $r$, with  a clear geometric interpretation.  
\end{abstract}

\maketitle

 \section{Introduction}
Let $C$ be a compact Riemann surface of genus $g \ge 2$. For a pair of integers $r$ and $d$, let $\mathcal{M}_{d}(r)$ denote the moduli space of semistable holomorphic vector bundles of rank $r$, degree $d$, and fixed determinant over $C$. When $r$ and $d$ are coprime, the stability and semistability conditions coincide, and  $\mathcal{M}_{d}(r)$ is a smooth projective variety. The intersection theory of these spaces has been studied extensively over the past several decades \cite{AtiyahBott, JeffreyK, Wittenrevisited, NS, HN74}. 

In contrast, when $r$ and $d$ are not coprime -- notably for $d=0$ -- the moduli space $\mathcal{M}_{0}(r)$ develops singularities coming from strictly semistable bundles. In this case, ordinary cohomology no longer satisfies Poincar\'e duality, and should be replaced by intersection cohomology. The intersection cohomology groups $IH^*(\mathcal{M}_0(r),\mathbb{Q})$, introduced by Goresky and MacPherson \cite{GM80, GM83}, restore this duality and provide the natural framework for extending the foundational results of Atiyah and Bott \cite{AtiyahBott}.

The study of the intersection theory of these singular moduli spaces has a long history. Classical approaches typically handle the singularities either through explicit desingularizations or via equivariant methods. An early result in this direction was obtained by Kiem \cite{Kiem}, who computed the intersection Poincar\'e--Verdier pairing on the moduli space $\Ms_0(2)$ of semistable bundles of rank $r=2$. This was later generalized by Jeffrey, Kirwan, Kiem, and Woolf \cite{JKKW06}, who gave a formula for the intersection pairing on $\Ms_0(r)$ for arbitrary rank $r\geq 2$.

The approach of \cite{JKKW06} is based on the infinite-dimensional viewpoint of Atiyah and Bott, in which the moduli space $\Ms_0(r)$ is realized as a geometric invariant theory quotient  $M\!\sslash\!G$. Using techniques developed in \cite{JKKW03}, the authors construct a partial desingularization $\widetilde{M}\!\sslash\!G$ via a sequence of blow-ups, producing a partial resolution of $\Ms_0(r)$. A key feature of their construction is that the intersection cohomology  $IH^*(\Ms_0(r),\mathbb{Q})$ embeds naturally into the ordinary cohomology  $H^*(\widetilde{M}\!\sslash\!G,\mathbb{Q})$, and that the intersection pairing on $IH^*(\Ms_0(r),\mathbb{Q})$ can be recovered from the pairing on the resolved space. Although this method is explicit, it becomes increasingly difficult to work with as the rank grows: the stratification of the resolution becomes more intricate, and explicit computations in $H^*(\widetilde{M}\!\sslash\!G,\mathbb{Q})$ for $r>2$ difficult to carry out in practice.

The goal of this paper is to give a complete and explicit description of the intersection pairing on $IH^*(\mathcal{M}_0(r),\mathbb{Q})$ for arbitrary rank $r$. Our approach avoids classical desingularization by exploiting the geometry of moduli spaces of stable parabolic bundles, which provide a natural finite-dimensional smooth model associated with $\Ms_0(r)$. More precisely, following \cite{ParCamOTASz}, we consider the smooth moduli space $\Ps_0(r)$ of stable parabolic bundles and the proper morphism $\pi\colon \Ps_0(r)\to \Ms_0(r)$, that forgets the parabolic structure. Unlike the blow-up constructions in \cite{JKKW06}, this map relates $\Ms_0(r)$  to a smooth space without requiring an explicit resolution of its singularities.

The key observation, based on the theory of perverse sheaves and the Decomposition Theorem of Beilinson, Bernstein, Deligne, and Gabber \cite{DT}, is that the intersection cohomology $IH^*(\Ms_0(r),\mathbb{Q})$ can be identified with a canonical subspace of the ordinary cohomology $H^*(\Ps_0(r),\mathbb{Q})$. As a result, the Poincar\'e--Verdier pairing on the singular moduli space $\Ms_0(r)$ may be computed by restricting the usual intersection pairing on the smooth space $\Ps_0(r)$. 

This description becomes explicit  through a tautological Hecke correspondence, established in \cite{OTASz}. Using this geometric construction,  the space $\mathcal{P}_0(r)$ can be realized as a projective bundle over the smooth moduli space $\mathcal{M}_1(r)$. The projective bundle formula then reduces the computation of intersection numbers on $\mathcal{P}_0(r)$ to intersection numbers on $\mathcal{M}_1(r)$, which can be evaluated using the iterated residue formulas of Jeffrey and Kirwan \cite{JeffreyK}.

\subsection{Remarks} Our finite-dimensional approach provide a useful framework for a deeper understanding of the singular moduli spaces $\Ms_0(r)$. In particular, using the identification of $IH^*(\Ms_0(r), \mathbb{Q})$ with a subspace of the cohomology ring $H^*(\Ps_0(r), \mathbb{Q})$ (see Proposition \ref{prop:lowestperv}), it should be possible to describe explicitly the intersection-cohomology fundamental class of $\Ms_0(r)$. This would be a first step toward constructing a full basis of $IH^*(\Ms_0(r), \mathbb{Q})$ in terms of universal classes on the moduli space of stable parabolic bundles $\Ps_0(r)$ -- so far, such a basis has been explicitly described only in rank $2$ \cite{Kiem}. 

It would also be interesting to compare these results with modern enumerative frameworks, such as Donaldson--Thomas invariants, to see whether the canonical subspace of $H^*(\Ps_0(r), \mathbb{Q})$ described in Proposition \ref{prop:lowestperv}  captures the same geometric information about $\Ms_0(r)$ as the virtual fundamental class, thus relating 
topological intersection theory and derived enumerative geometry.

\subsection{Contents of the paper\label{S:plan}} 
The paper is organized as follows. In \S\ref{S:IC} we recall the basic properties of intersection cohomology and the intersection pairing, that will be used throughout the paper. In particular, we formulate the Decomposition Theorem in a form suitable for our arguments.

In \S\ref{S:3}, we introduce the smooth moduli space of parabolic bundles $\Ps_0(r)$, which maps to the singular moduli space $\Ms_0(r)$ via a proper morphism $\pi\colon \Ps_0(r)\to \Ms_0(r).$ 
Using the results of \cite{ParCamOTASz}, we then identify the intersection cohomology $IH^*(\Ms_0(r),\mathbb{Q})$ with a canonical subspace of the cohomology ring $H^*(\Ps_0(r),\mathbb{Q})$ (see Proposition \ref{prop:lowestperv}).  Proposition \ref{prop:PDtoPV} further shows  that the computation of the Poincar\'e--Verdier pairing on $IH^*(\Ms_0(r),\mathbb{Q})$ can be reduced to the computation of the intersection pairing on $H^*(\Ps_0(r),\mathbb{Q})$. The next two sections are devoted to carrying out this latter computation.

In Propositions \ref{prop:H:10} and \ref{prop:bundle}, we show that the moduli space $\Ps_0(r)$ can be realized, via the Hecke correspondence, as a projective bundle over the smooth moduli space $\Ms_1(r)$. As a result, intersection numbers in $H^*(\Ps_0(r),\mathbb{Q})$ can be computed from the intersection pairing on $H^*(\Ms_1(r),\mathbb{Q})$ using the projective bundle formula. The latter pairing is evaluated using the formulas of Jeffrey and Kirwan \cite{JeffreyK}, which we recall in \S\ref{S:4}.

In \S\ref{S:5}, we introduce universal cohomology classes in $H^*(\Ps_0(r),\mathbb{Q})$ (see Theorem \ref{Thm:genP0}) and describe the cohomology ring $H^*(\Ps_0(r),\mathbb{Q})$ in terms of these classes. In Theorem \ref{Thm:intP0} we compute the integrals of these classes over $\Ps_0(r)$ using the results of \S\ref{S:4}, therefore obtaining explicit intersection numbers in $H^*(\Ps_0(r),\mathbb{Q})$. 

Finally, in \S\ref{S:6} we combine these results to obtain our main conclusion, stated in Theorem \ref{Thm:mainresult}: explicit formulas for the Poincar\'e--Verdier pairing on  $IH^*(\Ms_0(r),\mathbb{Q})$.  We end the paper by comparing our results with Kiem's computation of the intersection  pairing on $IH^*(\Ms_0(2),\mathbb{Q})$  \cite{Kiem} (cf. also \cite{JKKW06}) in \S\ref{S:Kiem}.

\subsection*{Acknowledgments} The authors gratefully acknowledge the advise and encouragement of Andr\'as Szenes. 
C.F. was supported by IN$\delta$AM GNSAGA and by the project: “Discrete Methods in Combinatorial Geometry and Geometric Topology” of the University of Modena and Reggio Emilia. O.T. was supported by the Austrian Science Fund (FWF) 10.55776/ESP2557325.

\section{Intersection cohomology and the Decomposition Theorem}\label{S:IC} 

\subsection{Intersection cohomology}
  We start with recalling some basic definitions and properties of \textit{intersection cohomology}, a topological invariant introduced by Goresky and MacPherson \cite{GM80, GM83} to extend Poincar\'e duality and intersection theory for singular varieties. For more details and proofs we refer to \cite{dCM09, GM83}. 
  
\noindent  
\textbf{Convention:} Throughout this paper, all (intersection) cohomology groups are taken with rational coefficients.

 Let $Y$ be a complex (possibly singular) quasi-projective algebraic variety. 
 The rational intersection cohomology groups of $Y$ form a graded vector space $IH^*(Y)=\oplus_{i\in\Z}IH^i(Y)$ which preserves many of the fundamental features of the cohomology of smooth varieties. In particular, intersection cohomology satisfies the Hard Lefschetz theorem, carries the Hodge package, and admits a non-degenerate Poincar\'e--Verdier intersection pairing.
    
 Just as orinary cohomology can be described 
as the hypercohomology of the shifted constant sheaf,  
$$H^{i}(Y)=\mathbb{H}^{i-\dim Y}(Y,\mathbb{Q}_Y[\dim (Y)]),$$ 
 intersection cohomology is defined as the hypercohomology of a cohomologically constructible complex, called the \textit{intersection cohomology sheaf} $IC_Y$:
 $$IH^i(Y):= \mathbb{H}^{i-\dim (Y)}(Y,IC_Y).$$
The complex $IC_Y$ restricts on the smooth locus $Y_{\mathrm{reg}}$ to the shifted constant sheaf: $$IC_Y\big|_{Y_{\mathrm{reg}}} = \mathbb{Q}_{Y_{\mathrm{reg}}}[\dim (Y)],$$ and thus can be viewed as an extension of the constant sheaf from the smooth part of $Y$ to the whole variety.
 
 More generally, there is a canonical construction, known as the  \textit{intermediate extension}, which associates to any local system $\mathcal{L}$ on a locally closed nonsingular subvariety $Z\subset Y$ a constructible complex $IC(\overline{Z},\mathcal{L})$ supported on the closure $\overline{Z}$. This complex extends $\mathcal{L}$ in a minimal way, in the sense that $$IC(\overline{Z},\mathcal{L})|_Z=\mathcal{L},$$
 and it satisfies some support and co-support conditions (see \cite{GM83}). 
 For such a pair $({Z},\mathcal{L})$  we define
 $$IH^i(\overline{Z},\mathcal{L}):= \mathbb{H}^{i-\dim (Y)}(\overline{Z},IC(\overline{Z},\mathcal{L})).$$

 One of the most powerful tools in the theory of intersection cohomology is the Decomposition Theorem \cite{DT}, which we formulate in the following subsection.

\subsection{The Decomposition Theorem}
In order to state the Decomposition Theorem in a form suitable for our purposes, we first recall some basic facts about the perverse filtration.
Throughout this section, we work in the following setup.
\begin{itemize}[leftmargin=12pt, itemsep=0pt, topsep=2pt]
\item Let $f \colon X \to Y$ be a proper surjective morphism of projective varieties, where $X$ is nonsingular of dimension $\nn$. 
\end{itemize}

To a morphism $f$ one can associate a canonical increasing filtration on $H^*(X)$, called the \emph{perverse filtration}. This filtration originates in the theory of perverse sheaves and admits a definition via the perverse truncation functors in the derived category $D^b_c(Y)$ of bounded constructible complexes, see e.g. \cite{BBD82, dCM05,dCM09,W17}.

While this construction justifies the canonical nature of the perverse filtration, it is often difficult to handle directly.  Under our simplifying assumption that $Y$ is projective, however, the perverse filtration admits a much more concrete description in terms of an ample class on the base, see \cite[Proposition 5.2.4]{dCM05}.

\begin{definition}\label{def:pervfiltr}
Fix an ample class $\beta\in H^2(Y)$ on $Y$, and consider the nilpotent operator 
$$ A\colon H^*(X) \rightarrow H^{*+2}(X)$$
defined by cupping with $f^* \beta \in H^2(X)$.
The \textit{perverse filtration} $P_\bullet$ on $H^d(X)$ is defined by\footnote{We adopt shifts so that $P_\bullet$ ranges from 0 to the (real) relative dimension of $f$.}
\[
P_k\mathrm{H}^d(X) \overset{\mathrm{def}}= \sum_{i\geq 1} \left(
\mathrm{ker}(A^{\nn+k+i-d}) \cap \mathrm{im}(A^{i-1}) \right) \cap \mathrm{H}^d(X).
\]
We denote the associated graded pieces 
by $$ H^d_k(X):=P_kH^d(X)/P_{k-1}H^d(X),$$
 and call them the \textit{perverse cohomology groups}.
\end{definition}

The following fundamental result of Beilinson, Bernstein, Deligne, and Gabber \cite{DT} states that the cohomology of $X$ admits a decomposition into the perverse cohomology groups. 

\begin{theorem}[Decomposition Theorem]\label{Thm:DT}
Let $f:X\rightarrow Y$ be a proper algebraic map from a nonsingular projective variety $X$ of dimension $\nn$ to a projective variety $Y$. Denote by $r(f) := \dim(X\times_YX)-\dim(X)$ the defect of semismallness of $f$. 
\begin{enumerate}[label=(\roman*)]
\item There is an isomorphism\footnote{In \cite{dCM07}, De Cataldo and Migliorini show  that the splittings in
\eqref{eq:dt} and \eqref{eq:dtglobal} are canonical once a class 
$\eta \in H^2(X)$ that is relatively ample with respect to $f$ is fixed.} in the derived category $\mathbf{D}^b_c(Y)$: 
\begin{equation}\label{eq:dt}
Rf_* \mathbb{Q}_X [\nn-r(f)] \simeq \bigoplus_{i=0}^{2r(f)} {}^p \mathcal{H}^i [-i],
\end{equation}
where the sheaves ${}^p \mathcal{H}^i := {^p\HH^i}(Rf_*\mathbb{Q}_X[\nn-r(f)])$ 
are called the \textit{perverse cohomology sheaves}. 
\item Taking hypercohomology on both sides of \eqref{eq:dt}  
yields isomorphisms
\begin{equation}\label{eq:dtglobal}
H^d(X)\cong \bigoplus_{i=0}^{2r(f)} H^{d-\nn + r(f)-i}(Y, {}^p \mathcal{H}^i), \text{ for all } 0\leq d\leq 2\nn.
\end{equation}
\item Under the isomorphism \eqref{eq:dtglobal}, the perverse filtration can be described as
\[
P_kH^d(X)=\mathrm{Im}\left\{H^{d-\nn + r(f)}(Y, \bigoplus_{i=0}^k{}^p \mathcal{H}^i[-i])\to H^d(X)\right\},
\]
and hence
\begin{equation*}
H_k^{d} (X) \simeq  H^{d- \nn + r(f)-k}(Y, {}^p \mathcal{H}^k).
\end{equation*}
\item There exists a stratification $Y=\sqcup_\alpha Y_\alpha$ of $Y$ adapted to $f$ such that, for each $i=0,\ldots, 2r(f)$, the perverse cohomology sheaves decompose as a direct sum of the intersection cohomology sheaves 
\begin{equation}\label{eq:dtlocsyst}
^p \mathcal{H}^i\cong \bigoplus_\alpha IC(\overline{Y}_\alpha,\mathcal{L}^i_\alpha).
\end{equation}
\end{enumerate}
Here $\mathcal{L}^i_\alpha$ are (shifted) local systems on $Y_\alpha$, defined by 
\[\mathcal{L}^i_\alpha=\mathcal{H}^{-\dim (Y_{\alpha})}(Y_\alpha,i_{\alpha}^* {}^p \mathcal{H}^i)[\dim (Y_\alpha)-i], \;\; i_\alpha:Y_\alpha\hookrightarrow Y.\] 
\end{theorem}
\noindent\textbf{Notation:} We say that a stratum $Y_\alpha$ \textit{appears as a support} in the decomposition \eqref{eq:dtlocsyst} of ${}^p \mathcal{H}^i$ if the corresponding summand $IC(\overline{Y}_\alpha,\mathcal{L}^i_\alpha)$ is nonzero.

Note that by construction of the local systems $\mathcal{L}^i_\alpha$, for any point $y\in Y_\alpha$ the stalk $(\mathcal{L}^i_\alpha)_y$  is a subspace of $H^{\dim(X)-\dim (Y_\alpha)+i-r(f)}(f^{-1}(y))$. As a consequence, we obtain the following statement.
\begin{corollary}\label{cor:stratasupp}
 If $r(f)$ coincides with the relative dimension of $f$, then a stratum $Y_\alpha$ can appear as a support in the decomposition \eqref{eq:dtlocsyst} of ${}^p \mathcal{H}^i$ only if 
$$\mathrm{codim} (Y_\alpha\subset Y)+i\leq 2\dim (f^{-1}(y)), \text{\; for \;}  y\in Y_\alpha.$$
\end{corollary}

\subsection{Perverse cohomology groups}
We finish this section by recalling several key properties of the perverse sheaves and the
perverse cohomology groups that will be used later. 

Let $\eta\in H^2(X)$ be a class that is relatively ample with respect to $f$, and let $k\in\mathbb{N}$. In the derived category $D^b_c(X)$, the class $\eta^k$ corresponds to a morphism $\mathbb{Q}_X\to \mathbb{Q}_X[2k]$. Applying the derived pushforward functor $Rf_*$  and passing to perverse cohomology, we obtain induced morphisms 
$$ {}^p\HH^{i}\to {}^p\HH^{i+2k}, \text{\; for \;} i\in\mathbb{Z}.$$
The relative Hard Lefschetz theorem states that these morphisms are isomorphisms in appropriate degrees. 
\begin{theorem}[Relative Hard Lefschetz]\label{Thm:HardLef} For every integer $k\geq 0$:
\begin{enumerate}[label=(\roman*)]
\item The map induced by cupping with $\eta^k$ in perverse cohomology is an isomorphism:
\begin{equation}\label{eq:RHLsheaves}
^p\HH^{r(f)-k}\xrightarrow{\simeq} {^p\HH^{r(f)+k}}.
\end{equation}
\item  The isomorphism \eqref{eq:RHLsheaves} is compatible  with the decomposition \eqref{eq:dtlocsyst}. In particular,  the set of strata $Y_\alpha$ appearing as supports in the decomposition of ${}^p \mathcal{H}^{r(f)-k}$ coincides with the set of those appearing in the decomposition of ${}^p \mathcal{H}^{r(f)+k}$. 
\item   Taking hypercohomology on both sides of \eqref{eq:RHLsheaves} yields isomorphisms 
$$\cup\eta^k\colon H_{r(f)-k}^{d}(X)\xrightarrow{\simeq} H_{r(f)+k}^{d+2k}(X), \text{ for all } d\geq 0.$$
    \end{enumerate} 
\end{theorem}

We next recall how Poincar\'e duality on $H^*(X)$ interacts with the Decomposition Theorem \ref{Thm:DT}. 
\begin{theorem}[Poincar\'e-Verdier duality] \label{thm:pvdual} 
Under the isomorphism $\eqref{eq:dtglobal}$, the Poincar\'e duality on $H^*(X)$ induces, for all $k,d\in \mathbb{Z}$,  isomorphisms
$$H_{r(f)-k}^{\nn-d}(X)\xrightarrow{\simeq} H_{r(f)+k}^{\nn+d}(X).$$
\end{theorem}

\begin{remark}
Although the Poincar\'e--Verdier duality on perverse cohomology and the duality induced by Hard Lefschetz are, a priori, of different nature, a crucial feature of our particular situation is that they coincide for certain specific perversities. This identification will play a central role in our computation of the intersection pairing on  $IH^*(\Ms_0(r))$. 
\end{remark}

\section{A canonical embedding of $IH^*(\Ms_0(r))$ into $H^*(\Ps_0(r))$}\label{S:3}

Our main goal in this section is to identify the intersection cohomology $IH^*(\Ms_0(r))$ with a canonical subspace of the cohomology ring $H^*(\mathcal{P}_0(r))$ of a certain moduli space of parabolic bundles; see Proposition \ref{prop:lowestperv}. We then show in  Proposition \ref{prop:PDtoPV} that the intersection pairing on  $IH^*(\Ms_0(r))$ can be expressed in terms of the intersection pairing on  $H^*(\Ps_0(r))$.

\subsection{Moduli spaces of parabolic bundles}\label{S:2}
We now introduce the relevant moduli space of stable parabolic bundles and recall its basic properties. For definitions and further details, see \cite{OTASz, ParCamOTASz}. 

Set  $c_1=1/2r$, $c_2=(1-r)/2r$; we denote by $$\Ps_0(r) \overset{\mathrm{def}}= \Ps_0((c_1,c_1,...,c_1,c_2))$$ the moduli space of rank-$r$, degree-$0$ parabolic bundles that are stable with respect to parabolic weights $(c_1,c_1,...,c_1,c_2)$. This moduli space has dimension $(r^2-1)(g-1)+r-1$, and for a bundle $W\in \Ps_0(r)$, its parabolic structure is given by the choice of a single line $F_1$ in the fiber $W_p$.

	Similarly, we define  $$\Ps_{1}(r) \overset{\mathrm{def}}= \Ps_1((c_2+1, c_1,...,c_1))$$ to be the moduli space of rank-$r$, degree-$1$ parabolic bundles that are stable with respect to parabolic weights $(c_2+1,c_1,...,c_1)$. In this case, the parabolic structure of a bundle $W\in\Ps_1(r)$ is given by the choice of a codimension-one subspace $F_{r-1}$ in $W_p$, and this moduli space has the same dimension as $\Ps_0(r)$.	
 
 The moduli spaces $\Ps_0(r)$ and $\Ps_1(r)$ satisfies the following properties, which are easy to prove.

 \begin{proposition}\label{prop:bundle}
 \begin{enumerate}[label=(\roman*)]
\item  Let $(W,F_1\subset W_p)$ be a stable parabolic bundle which represents a point in $\Ps_0(r)$. Then the vector bundle $W$ is semistable, and the corresponding map 
	\begin{equation}\label{def:pi}
	\pi\colon \Ps_0(r)\to \Ms_0(r)
	\end{equation}
	 is a proper morphism of algebraic varieties. 
	
\item Let $(W,F_{r-1}\subset W_p)$ be a stable parabolic bundle which represents a point in $\Ps_1(r)$. Then the vector bundle $W$ is stable, and thus represents a point in $\Ms_1(r)$. Moreover, $\Ps_1(r)$ is isomorphic to the projectivization 
	\begin{equation*}
	 \mathbf{P}(\mathcal{U}\big|_{\Ms_1(r)\times\{p\}}) \overset{f}{\to} \Ms_1(r)
	\end{equation*}
	 of the restriction of a universal bundle $\mathcal{U}$ over $\Ms_1(r)\times C$,  where we use the notation
$\mathbf{P}(E)\simeq\mathbb{P}(E^*)$ for the space of one-dimensional \emph{quotients} of $E$. 
\end{enumerate}
	\end{proposition}
    
\begin{remark}\label{rem:P1iso}
The projective bundle  $\mathbf{P}(\mathcal{U}\big|_{\Ms_1(r)\times\{p\}})$ is independent of the choice of the universal bundle $\mathcal{U}$ on $\Ms_1(r)\times C$. 
\end{remark}	

 In \cite[\S7]{OTASz} we introduced the tautological Hecke correspondence between moduli spaces of parabolic bundles with different degrees and parabolic  weights. In particular, the Hecke corrspondence provides an isomorphism between the two moduli spaces defined above.
   \begin{proposition}\cite[Proposition 7.1]{OTASz}\label{prop:H:10}
 Let $\Ps_1(r)$ and $\Ps_0(r) $ be the moduli spaces of degree-1 and degree-0 parabolic bundles defined above.  Then the tautological Hecke correspondence yields an isomorphism $\mathcal{H}\colon \Ps_1(r)\to \Ps_0(r)$. 
\end{proposition}

\subsection{Application of the Decomposition theorem to  $\pi\colon \Ps_0(r)\to \Ms_0(r)$}\label{S1.2} 

As we will explain, the identification of the intersection cohomology $IH^*(\Ms_0(r))$ with a subspace of the cohomology ring $H^*(\Ps_0(r))$ is obtained by applying the Decomposition Theorem \ref{Thm:DT} to the morphism $\pi \colon \mathcal{P}_0(r) \to \mathcal{M}_0(r),$ introduced in Proposition \ref{prop:bundle}. 
To prepare for this application, we describe a stratification of $\Ms_0(r)$, adapted to $\pi$. For further details, we refer the reader to \cite{ParCamOTASz}.

We write $\rho = [r_1,r_2,...,r_k] \vdash r$ for a partition of the integer r, which means  $$ r_1+\ldots+ r_k=r \text{\;\; and \;\;} r_1\ge r_2\ge\ldots\ge r_k > 0.$$ To each $V\in\Ms_0(r)$, we associate a partition $\rho(V)\vdash r$ as follows.

Recall that, via the Jordan--H\"older filtration, points of the moduli space $\Ms_0(r)$ correspond to isomorphism classes of \textit{polystable} bundles. For $V \in \Ms_0(r)$, let $\mathrm{Gr}(V)\simeq V_1\oplus V_2\oplus...\oplus V_k$ be the associated graded decomposition, where each  $V_i$ is stable of degree $0$. We denote by $\rho(V)$ the partition of $r$ given by the ranks of the summands $V_i$.

This construction induces a decomposition of $\Ms_0(r)$ into a disjoint union of locally closed subsets:
	\begin{equation*}
		\label{eq:partialstrat} \Ms_0(r) = \bigsqcup\limits_{\rho\vdash r} M_\rho,
	\end{equation*}
	where, for $ \rho =[r_1,\ldots,r_k] \vdash r$,
\[ M_\rho= \{ V_1\oplus\ldots\oplus V_k \in \mathcal{M}_0(r)\mid V_j- \, \text{stable, }\,   \mathrm{rk}(V_j)=r_j,\, \mathrm{deg}(V_j)=0, \,
    j=1,\ldots,k 
    \}/\cong.
\]

The above decomposition  is not a stratification in the usual sense, since the automorphism group of a polystable bundle $V_1\oplus \cdots \oplus V_k$ depends on whether some of the stable summands are isomorphic. This motivates the following definition.
\begin{definition}
A polystable bundle $V_1\oplus V_2\oplus...\oplus V_k$ is called \textit{abelian} if no two of its direct summands are isomorphic. 
\end{definition}
This notion induces a decomposition of each $M_\rho$, $\rho\vdash r$, as well as of $\Ms_0(r)$, into abelian and non-abelian parts:
$$ \Ms_0(r) = \Ms_0(r)^{\mathrm{ab}}\sqcup \Ms_0(r)^{\mathrm{nab}} \text{\;\; and \;\;} M_\rho = M_{\rho}^{\mathrm{ab}}\sqcup M_\rho^{\mathrm{nab}}.$$
The subsets $M_\rho^{\mathrm{ab}}$, for $\rho\vdash r$, are smooth quasi-projective varieties that form a partial stratification of the moduli space $\Ms_0(r)$ and satisfy $\overline{M_\rho^{\mathrm{ab}}} = \overline{M_\rho}$  (see \cite[Remark 3.3]{ParCamOTASz}). We note that when $\rho=[r]$ is the trivial partition, the corresponding stratum $M_{[r]}^{\mathrm{ab}}=M_{[r]}$  coincides with the smooth locus of $\Ms_0(r)$, consisting of stable bundles.

Using this partial stratification, we can summarize the results proved in \cite{ParCamOTASz}, that we will need in this paper.

\begin{theorem}\label{thm:dtpi}
    Let $\pi:\mathcal{P}_0(r)\rightarrow \mathcal{M}_0(r)$ 
    be the forgetful map described in Proposition \ref{prop:bundle}, and let $\rho=[r_1,r_2,...,r_k]\vdash r$ be a partition of $r$.
    \begin{enumerate}[label=(\roman*)]
    \item Let $V\in M^\mathrm{ab}_\rho$, then the fiber $\pi^{-1}(V)$  is a union of iterated projective bundle towers with   $$\dim(\pi^{-1}(V)) = \sum_{i<j}r_ir_j(g-1)+r-k.$$ In particular, when $V$ is stable (i.e. $V\in M_{[r]}$), the fiber is a single projective space: $\pi^{-1}(V)\simeq \mathbb{P}(V_p)\simeq \mathbb{P}^{r-1}.$
        \item The defect of semismallness (see Theorem \ref{Thm:DT}) of $\pi$ is $r(\pi)=r-1$. The non-abelian strata do not occur as supports of $\pi$, and the Decomposition Theorem for $\pi$ yields an isomorphism
        \begin{equation}\label{eq:dtsheaves}
        R\pi_*\mathbb{Q}_{\Ps_0(r)} [(r^2-1)(g-1)]= \bigoplus_{i=0}^{2r-2}\bigoplus\limits_{\rho\vdash r}
IC\left(\overline{M^\mathrm{ab}_\rho},\mathcal{L}^i_\rho\right)[-i], \end{equation}
        where $\rho$ runs over all partitions of $r$, and for a point $V\in M^\mathrm{ab}_\rho$, the stalk of the local system $\mathcal{L}^i_\rho$ at $V$  is a subspace of $H^{\dim(\Ms_0(r))-\dim(M_\rho)+i}(f^{-1}(V))$.
        \item The local system $\mathcal{L}_{[r]}^i$ on the smooth locus $M_{[r]}\subset \Ms_0(r)$ is trivial, with stalk isomorphic to $H^{i}(\mathbb{P}^{r-1})$.
    \end{enumerate}
\end{theorem}

\begin{remark}
In \cite{ParCamOTASz}, the results are stated for moduli spaces of degree-$0$ vector bundles with \emph{non-fixed determinant}, namely for the morphism
$\widetilde{\pi}\colon \widetilde{\Ps}_0(r)\to \widetilde{\Ms}_0(r).$
The determinant maps fit into the commutative diagram
$$
\begin{tikzcd}
\Ps_0(r) \arrow[hook]{r} \arrow{d}[swap]{\pi} &
\widetilde{\Ps}_0(r) \arrow{r}{\det} \arrow{d}[swap]{\widetilde{\pi}} & \mathrm{Jac}^0(C) \\
\Ms_0(r) \arrow[hook]{r} & \widetilde{\Ms}_0(r) \arrow{ur}[swap]{\det} &
\end{tikzcd}
$$
where the horizontal arrows identify the fixed-determinant moduli spaces  with the fibers of the determinant maps.
Consequently, the description of the fibers and the decomposition theorem for $\widetilde{\pi}$ in \cite{ParCamOTASz} apply line-by-line to the fixed-determinant case, with the only change that the dimensions of $\widetilde{\Ps}_0(r)$, $\widetilde{\Ms}_0(r)$, and the strata $\widetilde{M}_\rho$, for all $\rho\vdash r$, are reduced by $g$.
\end{remark}

\subsection{A canonical embedding of $IH^*(\Ms_0(r))$ into $H^*(\Ps_0(r))$}\label{S:CanEmbedding}

With the results established in the previous subsections, we can now identify the intersection cohomology $IH^*(\Ms_0(r))$ with a canonical subspace of $H^*(\Ps_0(r))$. More precisely, by Theorem \ref{thm:dtpi}, the defect of semismallness of $\pi$ is equal to $r-1$; consequently, the perverse filtration on $H^d(\mathcal{P}_0(r))$ associated with $\pi$  is concentrated in degrees from $0$ to  $2r-2$. In particular, for any $d\geq 0$ the zeroth perverce piece satisfies $$P_0H^d(\Ps_0(r)) = H_0^d(\Ps_0(r)),$$ defining a canonical subspace of $H^d(\Ps_0(r))$. In the following proposition, we show that this subspace is naturally isomorphic to $IH^d(\Ms_0(r))$.
\begin{proposition}\label{prop:lowestperv}
For any $d\geq 0$, we have canonical isomorphisms
\begin{equation}\label{eq:prop:lowestperv}
 IH^d(\Ms_0(r),\mathbb{Q})\simeq 
 H^d_0(\Ps_0(r),\mathbb{Q}). 
 \end{equation} 
\end{proposition}
\begin{proof}
By Corollary \ref{cor:stratasupp}, a stratum $M^\mathrm{ab}_\rho$ can appear as a support in the decomposition \eqref{eq:dtlocsyst} of ${}^p\mathcal{H}^{2r-2}$ only if 
\begin{equation*}\mathrm{codim}(M^\mathrm{ab}_\rho\subset \Ms_0(r))+2r-2\leq 2\dim(\pi^{-1}(V)), \text{ for } V\in M^\mathrm{ab}_\rho. 
\end{equation*}
A simple calculation shows that this inequality is satisfied only for $k\leq 1$. Consequently, the unique stratum appearing as a support in ${}^p\mathcal{H}^{2r-2}$, and hence, by the Relative Hard Lefschetz Theorem \ref{Thm:HardLef}(ii), also in  ${}^p\mathcal{H}^{0}$, is a smooth locus $M_{[r]}$. 

Hence we have  $$^\mathfrak{p}\mathcal{H}^0\simeq IC(\overline{M_{[r]}},\mathcal{L}_{[r]}^0) = IC(\mathcal{M}_0(r),\mathbb{Q}_{M_{[r]}}),$$ where  the last equality follows from Theorem \ref{thm:dtpi}(iii). 
Taking the hypercohomology in degree ${d-(r^2-1)(g-1)}$ yields the isomorphism \eqref{eq:prop:lowestperv}. 
\end{proof}

\noindent\textbf{Notation:} We denote by $i\colon IH^*(\Ms_0(r))\xrightarrow{\simeq} H^*_0(\Ps_0(r))$ the identification given by Proposition \ref{prop:lowestperv}\textcolor{red}.

\subsection{The Poincar\'e--Verdier pairing on $IH^*(\Ms_0(r))$}\label{S:PVpairing}

Recall that the aim  of this paper is to compute explicitly the Poincar\'e--Verdier pairing on the intersection cohomology $IH^*(\Ms_0(r))$. A key step in this calculation is provided by the following proposition, which relates  the ordinary
Poincar\'e pairing on the cohomology $H^*(\mathcal{P}_0(r))$ to the intersection pairing on $IH^*(\mathcal{M}_0(r)) $ via the identification of  $IH^*(\mathcal{M}_0(r))$ with the subset of
$H^*(\mathcal{P}_0(r))$, described in Proposition \ref{prop:lowestperv}.

\begin{proposition}\label{prop:PDtoPV}
Let $\mathcal{L}$ be a $\pi$-ample line bundle over $\Ps_0(r)$ such that $\mathrm{deg}(\mathcal{L}|_{\pi^{-1}(V)})=1$ for every stable bundle $V\in\Ms_0(r)$, and set $\eta=c_1(\mathcal{L})\in H^2(\Ps_0(r))$. 
For any two classes $\alpha, \beta\in IH^*(\Ms_0(r))$ satisfying $\mathrm{deg}(\alpha)+ \mathrm{deg}(\beta) = 2\mathrm{dim}(\Ms_0(r))$, define 
\begin{equation}\label{eq:lemma:concretedescr2}
 B(\alpha, \beta) := \int_{\Ps_0(r)}i(\alpha)\cup i(\beta)\cup \eta^{r-1}.
\end{equation}
Then  $B(\cdot,\cdot)$ coincides with the nondegenerate Poincar\'e--Verdier pairing on the intersection cohomology $IH^*(\Ms_0(r))$.
\end{proposition}
\begin{proof} 
Let $n := (r^2 - 1)(g - 1) = \dim (\Ms_0(r)).$ 
We first show that the bilinear form $B(\cdot,\cdot)$ is nondegenerate. Suppose that $\alpha \in IH^{2n-k}(\Ms_0(r))$ satisfies
$$ B(\alpha, \beta) = 0 \quad \text{for all } \beta \in IH^k(\Ms_0(r)). $$
We have to show that $\alpha = 0$.
Without loss of generality, we may assume $ k \leq n$. Then for any $\beta \in IH^k(\Ms_0(r))$, we can write
$$B(\alpha, \beta) = \langle i(\alpha), \eta^{r-1} \cup i(\beta)\rangle_{H(P)},$$ where $\langle\cdot,\cdot\rangle_{H(P)} $ denotes the Poincar\'e pairing in $H^*(\Ps_0(r))$. 

By the Relative Hard Lefschetz theorem, cup product with $\eta^{r-1}$ induces an isomorphism $$\cup \eta^{r-1}\colon H_0^{k}(\Ps_0(r))\xrightarrow{\sim}H_{2r-2}^{k+2r-2}(\Ps_0(r)).$$ Hence the vanishing of $ \langle i(\alpha), \eta^{r-1} \cup i(\beta)\rangle_{H(P)}$ for all $i(\beta)\in H_0^{k}(\Ps_0(r))$ is equivalent to  $$ \langle i(\alpha), \beta'\rangle_{H(P)}=0 \text{ \; for all\;  } \beta'\in H_{2r-2}^{k+2r-2}(\Ps_0(r)).$$ It follows from Theorem \ref{Thm:HardLef} that in this case $i(\alpha)=0$, and since $i$ is injective, we conclude that $\alpha=0$.

It remains to identify the bilinear form $B(\cdot,\cdot)$ with the Poincar\'e--Verdier pairing $\langle\cdot,\cdot\rangle_{IH(M)}$ on $IH^*(\Ms_0(r))$. 
By \cite[Remark 1.1.3]{dCM07}, it suffices to show that these two nondegenerate pairings agree when restricted to the smooth locus $M_{[r]}\subset \Ms_0(r)$ of stable vector bundles. 
More explicitly, consider the diagram
\begin{equation}\label{cd:restr}
\begin{tikzcd}
H^*(\Ps_0(r)) \arrow{r}{J^*} & H^*(\pi^{-1}(M_{[r]}))  \\
IH(\Ms_0(r)) \arrow[swap]{u}{i}  \arrow{r}{j^*} & IH^*(M_{[r]})\simeq  H^*(M_{[r]}) \arrow{u}{\pi^*},
\end{tikzcd}
\end{equation}
where $j\colon M_{[r]}\hookrightarrow \Ms_0(r)$ and $J\colon \pi^{-1}(M_{[r]})\hookrightarrow \Ps_0(r)$ denote the open embeddings. 
Since the Decomposition Theorem \ref{Thm:DT} is compatible with restricting to dense open sets (see e.g. \cite[\S 9.3]{L19}), this diagram commutes. To prove Proposition \ref{prop:PDtoPV},
it is enough to show that for any classes $\alpha \in IH^{2n-k}(\Ms_0(r))$ and $\beta\in IH^{k}(\Ms_0(r))$, the following equality holds:
\begin{equation}\label{eq:restrcoincide}
\int_{M_{[r]}} j^*(\alpha)\cup j^*(\beta) = \int_{\pi^{-1}(M_{[r]})} J^*(i(\alpha)\cup i(\beta)\cup \eta^{r-1}).
\end{equation}

Over the smooth points $V\in M_{[r]}$, the fiber $\pi^{-1}(V)$ is isomorphic to $\mathbb{P}^{r-1}$, and its fundamental class is $\eta^{r-1}|_{\pi^{-1}(V)}$. 
By the assumptions on $\eta$, we have $\pi_*J^*(\eta^{r-1}) = 1\in H^0(M_{[r]})$. Therefore, by the projection formula we obtain 
\begin{multline*}
\int_{M_{[r]}} \!j^*(\alpha)\cup j^*(\beta) \!=\! \int_{M_{[r]}} \!j^*(\alpha)\cup j^*(\beta)\cup \pi_*J^*(\eta^{r-1}) \!=\! \\ \int_{\pi^{-1}(M_{[r]})}\! \pi^{*}j^*(\alpha)\cup \pi^*j^*(\beta)\cup J^*(\eta^{r-1}). 
\end{multline*}
Using the commutativity of the diagram \eqref{cd:restr}, we rewrite this expression as
\begin{equation*}
 \int_{\pi^{-1}(M_{[r]})} J^*(i(\alpha)\cup i(\beta)\cup \eta^{r-1}), 
\end{equation*}
and thus we arrive at \eqref{eq:restrcoincide}.
\end{proof}

\begin{remark}\label{rem:toJKKW1}
In \cite{dCM05}, de Cataldo and Migliorini show that the perverse filtration $P_\bullet$ on $H^*(\Ps_0(r))$ is compatible with the intersection form on $H^*(\Ps_0(r))$. More precisely, the induced pairing between  the graded pieces $H^*_i(\Ps_0(r))$ and $H^*_j(\Ps_0(r))$ vanishes unless $i+j=2r-2$. Using this  orthogonality property,  one can show that the power of the relatively ample class in Proposition \ref{prop:PDtoPV} may be replaced by the top Chern class $c_{r-1}(\F_1^*\otimes \mathcal{U}_{p}/\F_1)$ 
of the relative virtual tangent bundle of the morphism $\pi$, scaled by $1/r$. \end{remark}

Proposition \ref{prop:PDtoPV} thus reduces the problem of computing the Poincar\'e--Verdier pairing on $IH^*(\Ms_0(r))$ to computing the intersection pairing on $H^*(\Ps_0(r))$. In the next sections we perform these computations, following the strategy described in \S\ref{S:plan}.

\section{Jeffrey--Kirwan formulas for the intersection pairing in $H^*(\Ms_1(r))$}\label{S:4}
Our aim in this section is to present the iterated residue formulas of Jeffrey and Kirwan \cite{JeffreyK} for intersection pairings on the moduli space $\Ms_1(r)$, in a form adapted to our later applications. For the result, see Theorem \ref{ThmFKtoHam}. 
\subsection{Generators of the cohomology ring $H^*(\Ms_1(r))$}\label{S:genHM}
We begin by describing a system of generators for the cohomology ring $H^*(\Ms_1(r))$ of the moduli space $\Ms_1(r)$ of stable vector bundles of rank $r$, degree $1$, and fixed determinant on a smooth curve $C$. Following Atiyah and Bott \cite{AtiyahBott}, it can be constructed as follows.

Over the product \(\mathcal{M}_{1}(r)\times C\) there exists a universal vector bundle $\mathcal{U}$ 
of rank $r$, well-defined up to tensoring with the pullback of a line bundle from $\mathcal{M}_{1}(r)$.  
Although the universal bundle itself is not unique, the bundle 
$$\widetilde{U} \overset{\mathrm{def}}{=} \mathcal{U}\otimes \mathrm{det}(\mathcal{U})^{-\frac{1}{r}}$$ is canonical. We note that $c_1(\widetilde{U})=0$, and for $2\le k\le r$,  $1\le j\le 2g$ we
introduce the cohomology classes
$$
\atil_{k}\in H^{2k}(\mathcal{M}_{1}(r)), 
\qquad
\btil^{\,j}_{k}\in H^{2k-1}(\mathcal{M}_{1}(r)),
\qquad
\widetilde{f}_{k}\in H^{2k-2}(\mathcal{M}_{1}(r)),
$$
 by means of the K\"unneth decomposition
$$c_k(\widetilde{U}) = \atil_k\otimes 1 + \sum_{j=1}^{2g}\btil_k^j\otimes e_j + \ftil_k\otimes \omega\in H^*(\Ms_1(r))\otimes H^*(C),$$
where $\omega\in H^{2}(C)$ is the fundamental class of the curve and $\{e_{j}\}$ is a basis of $H^{1}(C)$, such that $e_ie_{i+g}=\omega$ for $1\leq 
i\leq g$, and all other intersection numbers $e_ie_j$ equal 0.   
According to \cite[Proposition 2.20]{AtiyahBott}, these classes generate the cohomology ring $H^{*}(\mathcal{M}_{1}(r))$.

The goal of this section is to present the formulas, which compute all intersection pairings  of the form
\begin{equation}\label{eq:JKgeneralint} 
\int_{\Ms_1(r)}\prod_{k=2}^{r}\atil_k^{m_k}{\ftil_k}^{n_k}\prod_{j_k=1}^{2g}(\btil_{k}^{j_k})^{l_{k,j_k}}.
\end{equation}

To do this, we introduce some additional notation in the next subsections. For more details, we refer the reader to \cite{Szimrn, OTASz, Unyu, JeffreyK}.

\subsection{Notation: Hamiltonian bases and  $\iber$}\label{S:NotationBases}

\begin{itemize}[leftmargin=*, itemsep=0pt, topsep=2pt]
\item Let $T\subset SU_r$ be a maximal torus. We represent the Cartan subalgebra $V=\mathrm{Lie}(T)$ of the Lie algebra $\mathrm{Lie}(SU_r)$ as the quotient vector space 
 $$ V=\R^r/\R(1,1,\dots,1). $$ 
 There is a natural pairing between $V$ and 
\[ V^*=\{a=(a_1,\dots,a_r)\in\R^r|\; a_1+\dots+a_r=0\}. \]

\item Let $ x_1,x_2,\dots,x_r $ be the coordinates on $ \R^r $; 
given $ a\in V^* $, we will write $\scp ax  $ for the 
linear function $ \sum_ia_ix_i $ on $ V $. We will 
sometimes denote this function simply by $ a $.

\item Let $\Lambda$  be the integer lattice  in the vector space $V^*$:
\[ \Lambda=\{\lambda=(\lambda_1,\dots,\lambda_r)\in\Z^r|\; 
\lambda_1+\dots+\lambda_r=0\}. \]
For  $ 1\le i\neq j\le r $, we define 
the 
element $ \alpha^{ij} =x_i-x_j $ in $ \Lambda $. Let
\[ \rr = \{\ \pm\alpha^{ij}  |\,1\le i< j\le r\}  \]
be the set of roots of the $A_{r-1}$ root system with the opposite roots identified.

\item A basic object of our approach is an \textit{ordered} linear basis $ \mathbf{B} $  of $ V^* $ consisting of the elements of $ \Phi$. We denote by 
$\Bases$ the set of all such bases:
\[ \Bases= \left\{\mathbf{B}=\left(\beta^{[1]},\dots,\beta^{[r-1]}\right) \in\Phi^{r-1}|\; \bb \text{ -- basis of }  V^*\right\} .\]
\end{itemize}
Among these, we will be primarily interested in a distinguished family of bases defined via the action of the permutation group $\Sigma_r$. 
\begin{itemize}[leftmargin=*, itemsep=0pt, topsep=2pt]
\item Note that $V^*$ carries a natural action of  $\Sigma_r$, permuting the coordinates $x_j$ for $j=1,2,...,r$, and this action preserves the root system $\rr$. For each permutation $  \sigma\in \Sigma_r $, we define 
\begin{equation*}\label{eq:defsigmaB}
\bb_\sigma = 
(\alpha^{\sigma(r-1),\sigma(r)},\alpha^{\sigma(r-2),\sigma(r-1)},\dots,
\alpha^{\sigma(1),\sigma(2)})
\in\Bases. 
\end{equation*}
\end{itemize}

\noindent\textbf{Notation:}\label{not:Ham}  For any  $m\in\{1,2,...,r\}$, we define $ \HH_m
=\{\bb_\sigma|\;\sigma\in\Sigma_r,\,\sigma(1)=m\}$, 
and refer to the set $\HH_m$ as \emph{Hamiltonian basis}.

\begin{itemize}[leftmargin=*, itemsep=0pt, topsep=2pt]
\item Given a basis $\mathbf{B}=\left(\beta^{[1]},\dots,\beta^{[r-1]}\right) \in \Bases$ of $V^*$, and an element $a\in V^*$, we define $[a]_\mathbf{B}\in\Lambda$ to be the unique element of $V^*$ satisfying $$[a]_\bb = a- \{a\}_{\bb}, \text{\;\; where \;\;} \{a\}_\bb\in\sum_{j=1}^{r-1}[0,1)\beta^{[j]}.$$

\item Any basis $ \bb\in\Bases $ induces an isomorphism $V^* \simeq V$, and we will write ${\check{\alpha}}_\bb$ for the image of $\alpha\in V^*$ under this isomorphism. When no ambiguity arises, we omit the subscript $\bb$ to simplify the notation.

\item For a holomorphic function $Q$ on $V$  and $\alpha\in V^*$, we write  $Q_{\check{\alpha}_\bb}$ for the corresponding directional derivative. 
Then, given $Q$ and $ \bb\in\Bases $, we fix a homomorphism 
$$\alpha \mapsto \exp(Q_{\check{\alpha}_\bb})$$ from the additive group $V^*$ to the multiplicative group of non-vanishing holomorphic functions on $V\otimes_{\R}\C$. 
\end{itemize}

 With these ingredients in place, let $\bb=\left(\beta^{[1]},\dots,\beta^{[r-1]}\right) \in \Bases$ be a basis of $V^*$, let $f$ be a meromorphic function defined in a neighborhood of $0$ in 
$V \otimes_{\R} \C$ with poles contained in the union of hyperplanes
\begin{equation}\label{eq:hyperplanes}
  \bigcup_{1\le i<j\le r}\{x|\;x_i-x_j=0\},
\end{equation}
   and let $$Q(x)=\mathrm{const}\cdot \sum_{i<j}(x_i-x_j)^2+ \textit{higher order terms}$$ be a holomorphic function defined near $0$ in $V\otimes_{\R}\C$. Then we define
  \begin{equation}\label{defiber}
	\iber_{\mathbf{B},Q}   \left[ f(x) \right](a)\overset{\mathrm{def}}=
\frac1{(2\pi i)^{r-1}} \int\displaylimits_{Z_\bb} 
\frac{ 	f(x)\exp (Q_{\check{a}_\bb})\;d Q_{\check{\beta}^{[1]}}	\wedge 
	\dots\wedge d{Q}_{\check{\beta}^{[r-1]}}\;}{(1-\exp({Q}_{\check{\beta}^{[1]}}))\;
	\dots(1-\exp({Q}_{\check{\beta}^{[r-1]}}))\;},
\end{equation}
where the naturally oriented cycle $ Z_{\bb} $
is given by
\[ Z_{\bb} = \{x\in 
V\otimes_{\R}\C:\,|\scp{\beta^{[j]}}x|\;=\varepsilon_j,\, 
j=1,\,\dots,r-1 \}\subset 
V\otimes_{\R}\C\setminus\bigcup_{1\le i<j\le r}\{x|\;x_i=x_j\} \]
for sufficiently small fixed real constants 
$ 0\le\varepsilon_{r-1}\ll\dots\ll \varepsilon_{1}$. 
Note that operator $\iber_{\bb, Q}$ defines a linear map from the space of meromorphic functions on $V\otimes_\mathbb{R}\C$ with poles contained in \eqref{eq:hyperplanes} to the space of polynomials on $V^*$.

\begin{remark}\label{rem:iBer1}
As explained in \cite[Remark 4.3]{OTASz}, the linear operator $\iber_{\mathbf{B}, Q}$ has the following computational property. 
For $i=1,...,r-1$ we define $y_i=\scp{{\beta}^{[i]}}{x}$ and write $f$, $Q_{\check{\beta}^{[i]}}$ and $Q_{\check{a}}$ in these coordinates: $f(x)=\hat{f}(y)$, $Q_{\check{\beta}^{[i]}}(x) = \hat{Q}_{\check{\beta}^{[i]}}(y)$, $Q_{\check{a}}(x) = \hat{Q}_{\check{a}}(y)$. Then 
\[  \iber_{\mathbf{B}, Q}   \left[ f(x) \right] (a) =
\underset{y_1=0}{\res}\dots\underset{y_{r-1}=0}{\res}\frac{ 
	\hat f(y)\exp( \hat{Q}_{\check{a}})\;d \hat{Q}_{\check{\beta}^{[1]}}	\wedge 
	\dots\wedge d\hat{Q}_{\check{\beta}^{[r-1]}} }{(1-\exp(\hat{Q}_{\check{\beta}^{[1]}}))\;
	\dots(1-\exp(\hat{Q}_{\check{\beta}^{[r-1]}}))\;},
\]
where \textit{iterating} the residues means that at each step  we keep the variables with lower indices as unknown constants.
\end{remark}
We also observe the following simple facts, which will be used later.
\begin{remark}\label{rem:iBer2}
For any vector $v$ from the integer lattice $\Lambda$ in $V^*$ we have:  $\iber_{\bb, Q}   \left[ f(x) \right] (a+v) =  \iber_{\bb, Q}[ f(x)\exp(Q_{\check{v}_\bb}) ] (a)$.
\end{remark}
\begin{remark}\label{rem:iBer3}
For any permutation $\sigma\in\Sigma_r$, the change of variables $x=\sigma(y)$ yields the identity $\iber_{\sigma\bb}[f(x)](\sigma a-[\sigma c]_{\sigma\bb})= \iber_{\bb}[\sigma\cdot f(x)](a-[c]_\bb)$.
\end{remark}

\subsection{Notation: cohomology of the maximal torus $T\subset SU_r$}\label{S:cohTorus}  
Since the Jeffrey--Kirwan iterated residue formulas for the intersection pairings \eqref{eq:JKgeneralint} require integrating over $T^{2g}$ (see Theorem \ref{JKtheorem}), we need to introduce  additional notation for the cohomology classes in $H^*(T^{2g})$.

Choose an orthonormal basis $(u_1,u_2,...,u_{r-1})\in\Bases$, and let $\check{u}_1, \check{u}_2,..., \check{u}_{r-1}$ denote the dual basis of $V=\mathrm{Lie}(T)$. Write the Maurer-Cartan form $\theta\in \Omega^1(T)\otimes\mathrm{Lie}(T)$ in this basis as $$\theta = \sum_{a=1}^{r-1}\theta_a \check{u}_a.$$ 
The components $\theta_a$ then form a set of generators of $H^1(T)$, and according to \cite[Lemma 10.7]{JeffreyK}, we have the identity
\begin{equation*}
\int_{T} \theta_1\wedge \theta_2\wedge...\wedge \theta_{r-1}= \sqrt{r}.
\end{equation*}
Next, we define corresponding generators on $T^{2g}$. For each projection
$$\pi_j: T^{2g}\to T, \;\;\; j=1,2,...,2g$$
onto the $j^{\mathrm{th}}$ factor, set 
$$\zeta^j_a = \pi_j^*\theta_a, \;\;\; a=1,2,...,r-1.$$
The forms $\theta^j_a$ together provide a natural system of generators for $H^1(T^{2g})$, and (see \cite[Lemma 10.10]{JeffreyK}) we have:
\begin{equation}\label{eq:intTr}
\int_{T^{2g}} \exp\Big( \sum_{j=1}^g\sum_{a=1}^{r-1} \zeta_a^j\wedge\zeta_a^{g+j} \Big)= r^g.
\end{equation}

\subsection{Jeffrey--Kirwan iterated residue formulas}
With these preparations in place, we are  ready to present the formulas of Jeffrey and Kirwan, which describe all the intersection numbers on the moduli space $\Ms_1(r)$.

Let
 $\tau_k(x)$, for $k=2,...,r$, denote  the symmetric polynomials on $V\otimes_\R\C$ associated to the Chern classes $c_2,..., c_r$ of  the vector bundle $\widetilde{U}$ (see \S\ref{S:genHM}).  In particular, one has 
\begin{equation}\label{eq:tau2}
\tau_2(x)=\sum_{i<j}\Big(x_i- \frac{1}{r}{\sum_{a=1}^rx_a}\Big)\Big(x_j- \frac{1}{r}{\sum_{a=1}^rx_a}\Big) = -\frac{1}{2r}\sum_{i<j}(x_i-x_j)^2
\end{equation}
which, up to normalization, coincides with the Killing form of $SU_r$. Using this notation, the intersection numbers in the cohomology ring $H^*(\Ms_1(r))$ can be computed explicitly, as stated in the following theorem.

\begin{theorem}\cite[Theorem 9.12(a)]{JeffreyK}\label{JKtheorem} 
For $2\leq k\leq r$ and $1\leq j\leq 2g$, let $\atil_k, \ftil_k, \btil^j_k\in H^*(\Ms_1(r))$ and  $\zeta^j_k\in H^1(T^{2g})$ denote the cohomology classes defined in \S\ref{S:genHM} and \S\ref{S:cohTorus}. Let $Q(x)$ be a symmetric polynomial on $V\otimes_\R\C$, expressed in terms of the elementary symmetric polynomials $\tau_k$ as 
\begin{equation*}
Q(x) = \sum_{k=2}^r\delta_k\tau_k(x), 
\end{equation*}
where $\delta_k$ are formal nilpotent parameters, such that $\delta_2\neq 0$.  Define ${c}=(1/r,...,1/r,(1-r)/r)\in V^*$ and $\bb = (\alpha^{1 2}, \alpha^{2 3},...,\alpha^{r-1\, r} )\in\Bases$; 
then we have 
\begin{multline}\label{eq:16}
\int_{\Ms_1(r)}\exp(\delta_2\ftil_2+\delta_3\ftil_3+...+\delta_r\ftil_r)\prod_{k=2}^r\atil_k^{m_k} \prod_{j_k=1}^{2g}(\btil_k^{j_k})^{l_{k,j_k}} =  \\
\frac{(-1)^{{r \choose 2}(g-1)}}{r!} \sum_{w\in \Sigma_{r-1}} 
\iber_{\mathbf{B},Q} \Biggl[ \frac{\prod_{k=2}^r\tau_k(x)^{m_k} \exp(Q_{w\check{{c}}_\bb})}{\prod_{i<j}(x_i-x_j)^{2g-2}\, \mathrm{det}(\mathrm{Hess}_\bb(Q))}  \\
\int_{T^{2g}}\exp\left(-\sum_{a,b}\sum_{j=1}^g\zeta^j_a\zeta^{j+g}_b Q_{\check{u}_a\,\check{u}_b} \right) 
\prod_{k=2}^r\prod_{j_k=1}^{2g}\left(\sum_a \zeta_a^{j_k}\tau_{k_{\check{u}_a}}\right)^{l_{k,j_k}}\Biggr](-[w{c}]_\bb),
\end{multline} 
where the first sum runs over the elements $w\in \mathrm{Stab}_{r}\simeq\Sigma_{r-1}\subset \Sigma_{r}$, $\mathrm{Hess}_\bb(Q)$ denotes the Hessian matrix:  $\mathrm{Hess}_\bb(Q)_{i j} = \frac{\partial}{\partial \beta^{[i]}} Q_{\check{\beta}^{[j]}_\bb},$ and $(\check{u}_1, \check{u}_2,..., \check{u}_{r-1})$ is an orthonormal basis of $V$. 
\end{theorem}
We recall the following point from  \cite[Remark 9.13]{JeffreyK}. 
\begin{remark} 
In Theorem \ref{JKtheorem}, the classes $\zeta^j_k$ arise from the components of the Maurer--Cartan form $\theta \in \Omega^1(T) \otimes \mathrm{Lie}(T)$
with respect to an orthonormal basis $\check{u}_1, \ldots, \check{u}_{r-1}$ of $V = \mathrm{Lie}(T)$. In principle, this orthonormal basis  may be replaced by an arbitrary basis, provided that the classes $\zeta^j_a$ are defined using that basis; in that case, the second derivatives $Q_{\check{u}_a \check{u}_b}$ in \eqref{eq:16} must be replaced by $\frac{\partial}{\partial u_a} Q_{\check{u}_b}.$
Nonetheless, working with an orthonormal basis is convenient for computations over the torus $T^{2g}$.
\end{remark}
\begin{remark}
A simple calculation shows that for any $\bb\in\Bases$ we have $[(1/r,...,1/r,(1-r)/r)]_\bb = [(1/2r,...,1/2r, (1-r)/2)r]_\bb$.
\end{remark}

We conclude this section by rephrasing Theorem \ref{JKtheorem} in a form convenient for our forthcoming calculations.
\begin{theorem}\label{ThmFKtoHam}
In the notation of Theorem \ref{JKtheorem}, we have:
\begin{multline}\label{eq:rephrJK}
\int_{\Ms_1(r)}\exp(\delta_2\ftil_2+\delta_3\ftil_3+...+\delta_r\ftil_r)\prod_{k=2}^r\atil_k^{m_k} \prod_{j_k=1}^{2g}(\btil_k^{j_k})^{l_{k,j_k}} =  \\
\frac{(-1)^{{r \choose 2}(g-1)}}{r!} \sum_{\bb\in\HH_n} 
\iber_{\mathbf{B},Q} \Biggl[ \frac{\prod_{k=2}^r\tau_k(x)^{m_k} \exp(Q_{\check{{c}}_\bb})}{\prod_{i<j}(x_i-x_j)^{2g-2}\, \mathrm{det}(\mathrm{Hess}_\bb(Q))}  \\
\int_{T^{2g}}\exp\left(-\sum_{a,b}\sum_{j=1}^g\zeta^j_a\zeta^{j+g}_b Q_{\check{u}_a\,\check{u}_b} \right) 
\prod_{k=2}^r\prod_{j_k=1}^{2g}\left(\sum_a \zeta_a^{j_k}\tau_{k_{\check{u}_a}}\right)^{l_{k,j_k}}\Biggr](-[c]_\bb),
\end{multline} 
where the first sum now runs over the elements $\bb$ of the Hamiltonian basis $\HH_n$ for any $n\in\{1,2,...,r\}$ (see notation on page \pageref{not:Ham}).\footnote{More precisely, one may sum over the elements $\bb$ of any \emph{diagonal basis}; see \cite[\S3.2]{OTASz} for the definition. For brevity, we restrict to Hamiltonian bases in this paper.}
\end{theorem}
\begin{proof}
It follows from \cite[Remark 4.6]{OTASz} that the integral on the left-hand side of \eqref{eq:rephrJK} is equal to 
\begin{multline}\label{eq:rephrJK1}
\frac{(-1)^{{r \choose 2}(g-1)}}{r!} \sum_{\bb\in\HH_r} 
\iber_{\mathbf{B},Q} \Biggl[ \frac{\prod_{k=2}^r\tau_k(x)^{m_k} \exp(Q_{\check{{c}}_\bb})}{\prod_{i<j}(x_i-x_j)^{2g-2}\, \mathrm{det}(\mathrm{Hess}_\bb(Q))}  \\
\int_{T^{2g}}\exp\left(-\sum_{a,b}\sum_{j=1}^g\zeta^j_a\zeta^{j+g}_b Q_{\check{u}_a\,\check{u}_b} \right) 
\prod_{k=2}^r\prod_{j_k=1}^{2g}\left(\sum_a \zeta_a^{j_k}\tau_{k_{\check{u}_a}}\right)^{l_{k,j_k}}\Biggr](-[c]_\bb).
\end{multline} 
By \cite{Szimrn} (cf. also \cite[Theorem 4.4]{OTASz}) the sum in \eqref{eq:rephrJK1} is independent on the choice of the Hamiltonian basis, hence we obtain the expression \eqref{eq:rephrJK}. 
\end{proof}

\section{Intersection pairing in $H^*(\Ps_0(r))$}\label{S:5}
Our goal in this section is to present an explicit description of the intersection pairing on the moduli space $\Ps_0(r)$ of degree-zero stable parabolic bundles. For the result, see  Theorem \ref{Thm:intP0}.

 \subsection{Generators of the cohomology ring  $H^*(\Ps_0(r))$}\label{S:genP0} 
 We begin by describing the generators of the cohomology ring $H^*(\Ps_0(r))$. To do so, we first introduce some additional notation. 
 
 Let $\mathcal{U}_0$ and $\mathcal{U}_1$ be universal bundles over $\Ps_0(r)\times C$ and $\Ps_1(r)\times C$, respectively. Denote by $\F_1\subset \F_2=\mathcal{U}_{0|p} := \mathcal{U}_0\big|_{\Ps_0(r)\times\{p\}}$ and $\G_1\subset \G_2=\mathcal{U}_{1|p} := \mathcal{U}_1\big|_{\Ps_1(r)\times\{p\}}$ the corresponding universal flag bundles.  We then define cohomology classes $\x\in H^2(\Ps_0(r))$ and $\xtil\in H^2(\Ps_1(r))$ as 
\begin{multline}\label{eq:xxtil}
\x = c_1(\F_1\otimes \mathrm{det}(\mathcal{U}_{0|p})^{-1/r})\in H^2(\Ps_0(r))  \text{\;\; and \;\;} \\
 \xtil = c_1(\G_2/\G_1\otimes \mathrm{det}(\mathcal{U}_{1|p})^{-1/r})\in H^2(\Ps_1(r)).
 \end{multline}
These classes are independent of the choice of universal bundles $\mathcal{U}_0$ and $\mathcal{U}_1$.  Later (cf. Lemma \ref{lemma:Ucorrespond}), we will show that the class $\x$ pulls back to $\xtil$ under the  Hecke correspondence $\mathcal{H}\colon \Ps_1(r)\to\Ps_0(r)$. 

We summarize the statements of Propositions \ref{prop:bundle} and \ref{prop:H:10}  in the following diagram:
\begin{equation}\label{diagram:paps}
\begin{tikzcd}
\Ps_0(r) \arrow[d, "\pi"]  \arrow[drr, "h"{yshift=-11pt}] & \arrow[l, "\simeq"{yshift=1pt}, "\mathcal{H}"{yshift=10pt}] \Ps_1(r)  \arrow[r, phantom, sloped, "\simeq"] & \mathbf{P}(\widetilde{U}_p)  \arrow[d,"f"] \\
\Ms_0(r) & & \Ms_1(r).
\end{tikzcd}
\end{equation}
Thus, following \cite[Proposition 2.20]{AtiyahBott}, the cohomology ring $H^*(\Ps_0(r))$ is generated by the pullbacks of the classes $\atil_k, \btil_k^j, \ftil_k \in H^*(\Ms_1(r))$, for $2\leq k\leq r$ and  $1\leq j\leq 2g$, defined in \S\ref{S:genHM} via the K\"unneth decomposition of the Chern classes of the bundle $\widetilde{U}\to \Ms_1(r)\times C$, together with the class $(\mathcal{H}^{-1})^*\xtil$.

Now define
\begin{equation}\label{eq:defUbar}
\overline{U} \overset{\mathrm{def}}{=} \mathcal{U}_0\otimes \mathrm{det}(\mathcal{U}_0)^{-\frac{1}{r}}.
\end{equation}  We note that $\overline{U}$ is independent on the choice of the universal bundle $\mathcal{U}_0$.
The goal of this subsection is to give a geometric interpretation of the above cohomology classes on $\Ps_0(r)$ in terms of the K\"unneth components of the Chern classes of the bundle $\overline{U}$;  see Theorem \ref{Thm:genP0} for the result.

We begin by showing the following key  relation between the Chern characters of the bundles $\overline{U}$  and $\widetilde{U}$. 
 \begin{proposition}\label{Prop:chUtildeUbar} 
In the notation introduced above, we have 
$$ch(\overline{U}) = (h\times\mathrm{id})^*ch(\widetilde{U}) - \omega  \exp(({\mathcal{H}^{-1})^*\xtil})  + \frac{1}{r}\omega\, h^*ch(\widetilde{U}_p) , $$
where  $\omega\in H^2(C)$ denotes the fundamental class of the curve, and $\xtil \in H^2(\Ps_1(r))$ is the class defined in \eqref{eq:xxtil}. 
  \end{proposition}
 \begin{proof}
 We introduce the notation $U_1$ for the universal bundle on 
$\Ps_1(r)\times C$, equipped with the universal flag
\[
    \G_1 \subset \G_2 = U_1\big|_{\Ps_1(r)\times\{p\}}\overset{\mathrm{def}}{=} U_{1|p},
\]
and normalized so that the quotient $\G_2/\G_1$ is a trivial line bundle.
Similarly, we denote by $U_0$ the universal bundle on 
$\Ps_0(r)\times C$,  with universal flag
\[
    \F_1 \subset \F_2 = U_0\big|_{\Ps_0(r)\times\{p\}}\overset{\mathrm{def}}{=} U_{0|p},
\]
normalized so that the subbundle 
$\F_1 \hookrightarrow  U_{0|p}$ is trivial. Note that since the expression \eqref{eq:defUbar}  is independent on the choice of the universal bundle over $\Ps_0(r)\times C$, we have 
\begin{equation*}\label{eq:trivialaboutUbar}
\overline{U} \simeq U_0\otimes\mathrm{det}(U_{0})^{-\frac{1}{r}}.
\end{equation*} The following Lemma is a simple observation (cf. \cite[\S7]{OTASz}). 
 \begin{lemma}\label{lemma:Ucorrespond}
Under the Hecke correspondence $\mathcal{H}$, the line bundle $\G_2/\G_1$  on $\Ps_1(r)$ corresponds to the line bundle $\F_1$ on $\Ps_0(r)$, and thus the normalized universal bundle $U_1$ corresponds to the universal bundle $U_0$. 
 \end{lemma}
Thus applying the operator $\mathcal{H}$ to the normalized universal bundle $U_1$
yields the following short exact sequence of the corresponding sheaves of sections. Note that here and in what follows, by a slight abuse of notation, we use the same symbols to denote both the vector bundles and their associated sheaves of sections.
 \begin{equation}\label{SESU} 0 \to (\mathcal{H}\times\mathrm{id})^*{U}_0\to {U}_1 \to  \mathcal{G}_2/\mathcal{G}_{1} \to 0. \end{equation}
Passing to the associated vector bundles gives the relation
\begin{equation}\label{chUch}
ch(U_1)=(\mathcal{H}\times\mathrm{id})^*ch(U_0)+\omega.
\end{equation} 

Next, tensoring the exact sequence \eqref{SESU} by $\mathrm{det}({U}_{1})^{-1/r} $, we obtain the short exact sequence
\begin{equation*} 0 \to (\mathcal{H}\times\mathrm{id})^*{U}_0\otimes \mathrm{det}({U}_{1})^{-1/r} \to {U}_1\otimes \mathrm{det}({U}_{1})^{-1/r} \to  \mathcal{G}_2/\mathcal{G}_{1}\otimes\mathrm{det}({U}_{1})^{-1/r} \to 0, \end{equation*}
which can be rewritten as 
\begin{equation}\label{SESUtilde}
0 \to (\mathcal{H}\times\mathrm{id})^*{U}_0\otimes \mathrm{det}({U}_{1})^{-1/r} \to (f\times\mathrm{id})^*\widetilde{{U}} \to  \mathcal{G}_2/\mathcal{G}_{1}\otimes\mathrm{det}({U}_{1})^{-1/r} \to 0. 
\end{equation}
Using equations \eqref{eq:xxtil}, \eqref{chUch} and the fact that $\omega\, ch({V}) =\omega\, ch({V}\big|_{X\times\{p\}})$ for any bundle $V$ over the product $X\times C$, we arrive at the following equality of the Chern characters:
\begin{multline}\label{eq:chtilde} 
(f\times\mathrm{id})^*ch(\widetilde{U}) - \omega\,\exp(\xtil)  \overset{\eqref{eq:xxtil}}{=} \\ (f\times\mathrm{id})^*ch(\widetilde{U}) - \omega\, ch(\mathrm{det}(U_{1|p})^{-1/r}) 
\overset{\eqref{SESUtilde}}{=} 
 (\mathcal{H}\times\mathrm{id})^*ch(U_0)\cdot ch(\mathrm{det}({U}_{1})^{-1/r})  = \\
(\mathcal{H}\times\mathrm{id})^*ch(U_0)\cdot \exp(-c_1({U}_{1})/r) = 
(\mathcal{H}\times\mathrm{id})^*ch(U_0)\cdot \exp(-c_1({U}_{1|p})/r -\omega/r) \overset{\eqref{chUch}}{=} \\
(\mathcal{H}\times\mathrm{id})^*ch(U_0)\cdot \exp(-c_1({U}_{0|p})/r -\omega/r) \overset{}{=} \\
(\mathcal{H}\times\mathrm{id})^*ch(U_0)\cdot  ch(\mathrm{det}({U}_{0})^{-1/r})(1-\omega/r) \overset{\eqref{chUch}}{=} 
(\mathcal{H}\times\mathrm{id})^*ch(\overline{U}) -\omega/r\cdot f^*ch(\widetilde{U}_p).
\end{multline}

Finally, pulling back both sides of \eqref{eq:chtilde} by $(\mathcal{H}^{-1}\times\mathrm{id})^*$, we arrive at the statement of Proposition \ref{Prop:chUtildeUbar}. 
\end{proof}

Our next goal is to relate the Chern classes of the bundles $\overline{U}$ and $\widetilde{U}$.  We begin by recalling and setting up the notation that we will use. 

\begin{itemize}[leftmargin=*, itemsep=0pt, topsep=2pt]
\item As above, we write $x_1,x_2,...,x_r$ for the standard coordinates on  $\mathbb{R}^r$, and set $V=\mathbb{R}^r/(1,...,1).$ For each index $i=1,2,...,r$, we denote $$\xb_i := x_i - \frac{1}{r}\sum_{j=1}^rx_j\in V^*.$$ 
\item Given any $\Sigma_r$-invariant polynomial $q$ on $V\otimes_\R\C$, we write 
\begin{equation}\label{q(U)def}
q(\overline{U})\in H^*(\Ps_0(r)\times{C}), \; q(\widetilde{U})\in H^*(\Ps_1(r)\times{C})
\end{equation}
 for the cohomology class obtained by evaluating $q(\xb_1,...,\xb_r)$ on the Chern roots of the bundles $\overline{U}$ and  $\widetilde{U}$, respectively. In particular, for the elementary symmetric polynomial $\tau_k$ of degree $k$, we have $$\tau_k(\widetilde{U}) = c_k(\widetilde{U}) \text{\;\; and \;\;} \tau_k(\overline{U}) = c_k(\overline{U}),$$ so $\tau_k$ recover the usual Chern classes of these bundles. 

\item We write $\mathbb{C}[\xb_1,...,\xb_r]^{\Sigma_r}$ for the ring of symmetric polynomials on $V\otimes_\R\C$,  and $\mathbb{C}[\xb_1,...,\xb_r]^{\Sigma_r}[\xb_1]$ for the ring of polynomials in the variable $\xb_1$ with coefficients in $\mathbb{C}[\xb_1,...,\xb_r]^{\Sigma_r}$.  
Under the natural identification of Chern roots, we have  $$\xb_1(\widetilde{U}_p) :=  \xb_1(\widetilde{U})\big|_{\Ps_1(r)\times\{p\}} = \xtil \text{\;\; (cf. \eqref{eq:xxtil})}.$$ 
\item Using this identification, for any polynomial $q\in\mathbb{C}[\xb_1,...,\xb_r]^{\Sigma_r}[\xb_1]$ we obtain a well-defined cohomology class 
$$q(\widetilde{U}_p) := q(\widetilde{U})\big|_{\Ps_1(r)\times\{p\}} \in H^*(\Ps_1(c)).$$
\end{itemize}
Later, we will use the following simple observation.
\begin{lemma}\label{l:derivx}
 Let  $q(\xb_1,...,\xb_r) \in \C[\xb_1,...,\xb_r]^{\Sigma_r}$. Let $\bb\in\Bases$ be a basis of $V^*$ (see \S\ref{S:NotationBases}),  and define ${s} := ((1-r)/r, 1/r,...,1/r)\in V^*$. 
Let $q_{\check{s}}$ denote the directional derivative of $q$, expressed in the variables corresponding to the basis $\bb$, in the direction $\check{s}_\bb$. 
The resulting polynomial $q_{\check{s}}$ is independent of the choice of $\bb\in\Bases$, and, moreover,
$$q_{\check{s}}  \in \C[\xb_1,...,\xb_r]^{\Sigma_r}[\xb_1].$$ 
\end{lemma}

Armed with these preparatory results, we can now describe a set of generators for the cohomology ring  $H^*(\Ps_0(r))$ in terms of the universal classes on $\Ps_0(r)$. 

\begin{theorem}\label{Thm:genP0}
Let $\mathcal{U}$ be a universal bundle on $\Ps_0(r)\times C$ and let  $\F_1\subset \mathcal{U}\big|_{\Ps_0(r)\times\{p\}} :=\mathcal{U}_p$ be the corresponding flag bundle. Recall the notation $\x= c_1(\F_1 \otimes \mathrm{det}(\mathcal{U}_p)^{-1/r})\in H^2(\Ps_0(r))$  and  $\overline{U}=\mathcal{U} \otimes \mathrm{det}(\mathcal{U})^{-1/r}$.
For $2\le k\le r$ and $1\le j\le 2g$,
define the cohomology classes
$$
a_{k}\in H^{2k}(\Ps_0(r)), 
\qquad
b^{\,j}_{k}\in H^{2k-1}(\Ps_0(r)),
\qquad
f_{k}\in H^{2k-2}(\Ps_0(r))
$$ 
via the K\"unneth decomposition
$$c_k(\overline{U}) = a_k\otimes 1 + \sum_{j=1}^{2g}b_k^j\otimes e_j + f_k\otimes \omega\in H^*(\Ps_0(r))\otimes H^*(C),$$
where $\omega\in H^{2}(C)$ is the fundamental class of the curve and $\{e_{j}\}$ is a basis of $H^{1}(C)$, such that $e_ie_{i+g}=\omega$ for $1\leq 
i\leq g$, and all other intersection numbers $e_ie_j$ equal 0. 
\begin{enumerate}[label=(\roman*)]
\item Then for any $2\le k\le r$ and $1\le j\le 2g$,  we have:
\begin{equation}\label{eq:Thm:genP0}
a_k = h^*(\atil_k), \qquad b_k^j = h^*(\btil_k^j) \text{\;\; and \;\;} f_k = h^*(\ftil_k) - h^*({\tau_k}_{\check{s}}(\widetilde{U}_p)),
\end{equation}
 where $h\colon \Ps_0(r)\to \Ms_1(r)$ is the morphism defined in \eqref{diagram:paps}, $\tau_k\in \C[\xb_1,...,\xb_r]^{\Sigma_r}$ denotes the elementary symmetric polynomial of degree $k$, and ${\tau_k}_{\check{s}}(\widetilde{U}_p)$ is the cohomology class introduced prior to Lemma \ref{l:derivx}. 
\item The cohomology ring $H^*(\Ps_0(r))$ is generated by the classes $a_{k}, b^{\,j}_{k}, f_{k}$ and $\x$. 
\end{enumerate}
\end{theorem}
\begin{proof}
We introduce the notation $q_i(y_1,...,y_r)=\frac{1}{i!}(y_1^i+...+y_r^i)$; in particular, we have 
$q_i(\widetilde{U})=ch_i(\widetilde{U})$ and $q_i(\overline{U})=ch_i(\overline{U})$. It follows from 
Proposition \ref{Prop:chUtildeUbar} that 
$$q_i(\overline{U})=(h\times\mathrm{id})^*(q_i(\widetilde{U}))-\omega \, h^*({q_i}_{\check{s}}(\widetilde{U}_p)),$$ and thus
\begin{multline}\label{chderprod}
q_i(\overline{U})q_i(\overline{U}) =  
(h\times\mathrm{id})^*(q_i(\widetilde{U})q_i(\widetilde{U})- \omega \, ( {q_i}_{\check{s}}(\widetilde{U}_p)q_i(\widetilde{U})+ {q_j}_{\check{s}}(\widetilde{U}_p)q_i(\widetilde{U}))) = \\
(h\times\mathrm{id})^*(q_i(\widetilde{U})q_i(\widetilde{U})) - \omega \,  h^*((q_iq_j)_{\check{s}}(\widetilde{U}_p)).
\end{multline}
For the last equality, we used the fact that $\omega \, ch(\widetilde{U})=\omega \, ch(\widetilde{U}_p)$.

Since any polynomial $q\in\mathbb{C}[y_1,...,y_r]^{\Sigma_r}$ can be written as a polynomial in $q_i$'s, \eqref{chderprod} implies that for any symmetric polynomial $q$ on $V\otimes_\R\C$ we have:
$$q(\overline{U})=(h\times\mathrm{id})^*(q(\widetilde{U}))-\omega\, h^*(q_{\check{s}}(\widetilde{U}_p)).$$
Substituting $q=\tau_k$ for $k=2,3,...,r$, we arrive at the statement of  Theorem \ref{Thm:genP0}(i).

It follows from Lemma \ref{l:derivx} that the polynomial ${\tau_k}_{\check{s}}$ is an element of $\C[\xb_1,...,\xb_r]^{\Sigma_r}[\xb_1]$, and therefore the cohomology class $({\tau_k}_{\check{s}})(\widetilde{U}_p)$ can be expressed as a polynomial in the classes $$\atil_j = \tau_i(\widetilde{U}_p) = \tau_i(\widetilde{U})\big|_{\Ms_1(r)\times\{p\}}, \;\;\;  j=2,3,...,r,$$ and $\xtil = \xb_1(\widetilde{U}_p)$. 

By Lemma  \ref{lemma:Ucorrespond}, we have the identification $\x = ({\mathcal{H}^{-1})^*\xtil}$, and hence the cohomology ring  $H^*(\Ps_0(r))$, which is generated by the classes $$h^*(\atil_k),\, h^*(\btil^j_k),\, h^*(\ftil_k) \;\; (k= 2,...,r; \, j=1,...,2g) \text{\; and \;}  (\mathcal{H}^{-1})^*\xtil,$$  may equivalently be described as generated by the corresponding classes $a_k, b^j_k, f_k$, and $\x$.  This concludes the proof of Theorem \ref{Thm:genP0}
\end{proof}

\subsection{Integration over $\Ps_0(r)$}
Now  we are ready to state the main result of this section, which computes all the intersection pairings on the moduli space $\Ps_0(r)$. 
\begin{theorem}\label{Thm:intP0}
Let $\x, a_k, f_k, b^j_k\in H^*(\Ps_0(r))$ and  $\zeta^j_k\in H^1(T^{2g})$, for $2\leq k\leq r$ and $1\leq j\leq 2g$, denote the cohomology classes defined in \S\ref{S:genP0} and \S\ref{S:cohTorus}. Recall the notation  $Q(x)$ for the  symmetric polynomial on $V\otimes_\R\C$, expressed in terms of the elementary symmetric polynomials $\tau_k$ as 
\begin{equation}\label{def:Q}
Q(x) = \sum_{k=2}^r\delta_k\tau_k(x),
\end{equation}
where $\delta_k$ are formal nilpotent parameters, such that $\delta_2\neq 0$.  Define ${c}=(1/r,...,1/r, 1/r-1)\in V^*$; 
then we have 
\begin{multline}\label{eq:ourint}
\int_{\Ps_0(r)}\exp(\delta_2f_2+\delta_3f_3+...+\delta_rf_r)\x^m \prod_{k=2}^ra_k^{m_k} \prod_{j_k=1}^{2g}(b_k^{j_k})^{l_{k,j_k}} =  \\
\frac{(-1)^{{r \choose 2}(g-1)}}{(r-1)!} \sum_{\bb\in \HH_n} 
\iber_{\mathbf{B},Q} \Biggl[ \frac{\big(\frac{1}{r}\sum_{j=1}^r(x_r-x_j)\big)^m \prod_{k=2}^r\tau_k(x)^{m_k}}{\prod_{i=1}^{r-1}(x_r-x_i)\prod_{i<j}(x_i-x_j)^{2g-2}\, \mathrm{det}(\mathrm{Hess}_\bb(Q))}  \\
\int_{T^{2g}}\exp\left(-\sum_{a,b}\sum_{j=1}^g\zeta^j_a\zeta^{j+g}_b Q_{\check{u}_a\,\check{u}_b} \right) 
\prod_{k=2}^r\prod_{j_k=1}^{2g}\left(\sum_a \zeta_a^{j_k}\tau_{k_{\check{u}_a}}\right)^{l_{k,j_k}}\Biggr](-[c]_\bb),
\end{multline} 
where the first sum runs over the elements $\bb$ of the Hamiltonian basis $\HH_n$ for any $n\in\{1,2,...,r\}$ (see notation on page \pageref{not:Ham}), $\mathrm{Hess}_\bb(Q)$ denotes the Hessian matrix:  $\mathrm{Hess}_\bb(Q)_{i j} = \frac{\partial}{\partial \beta^{[i]}} Q_{\check{\beta}^{[j]}_\bb},$ and $(\check{u}_1, \check{u}_2,..., \check{u}_{r-1})$ is an orthonormal basis of $V$. 
\end{theorem}

\begin{proof}
First, we rewrite the left-hand side of \eqref{eq:ourint} as an integral over the moduli space of degree-1 parabolic bundles $\Ps_1(r)$. 
Applying Theorem \ref{Thm:genP0}, we obtain
\begin{multline}\label{eq:intP01}
\int_{\Ps_0(r)}\exp(\delta_2f_2+\delta_3f_3+...+\delta_rf_r) \x^m \prod_{k=2}^ra_k^{m_k} \prod_{j_k=1}^{2g}(b_k^{j_k})^{l_{k,j_k}} =  \\ 
\int_{\Ps_1(r)}\exp(\delta_2\ftil_2+\delta_3\ftil_3+...+\delta_r\ftil_r) \prod_{k=2}^r\atil_k^{m_k} \prod_{j_k=1}^{2g}(\btil_k^{j_k})^{l_{k,j_k}} \exp(-Q_{\check{s}}(\widetilde{U}_p)) \xtil^m,
\end{multline}
where ${s} = ((1-r)/r, 1/r,...,1/r)\in V^*$ and as above,  $Q_{\check{s}}(\widetilde{U}_p)$ denotes the cohomology class defined prior to Lemma \ref{l:derivx}. 

To perform the integral over the moduli space $\Ps_1(r)$, we will use the following proposition.
\begin{proposition}\label{prop:intPtoM}
For any  $q\in \C[\xb_1,...,\xb_r]^{\Sigma_r}[\xb_1]$ and $\alpha\in H^*(\Ms_1(r))$, we have:
\begin{equation}\label{eq:intPtoM}
\int_{\Ps_1(r)}f^*(\alpha)  q(\widetilde{U}_p) =\int_{\Ms_1(r)} \alpha \, \left( \sum_{\sigma\in \Sigma_r/\mathrm{Stab}_{1}} \frac{\sigma\cdot q}{\prod_{j\neq \sigma(1)} (x_{\sigma(1)}-x_j)}\right)\left(\widetilde{U}_p\right)\!,
\end{equation}
where $f\colon \Ps_1(r)\to \Ms_1(r)$ is the projective bundle map (see Proposition \ref{prop:bundle}), and $$\left(\sum_{\sigma\in \Sigma_r/\mathrm{Stab}_{1}} \frac{\sigma\cdot q}{\prod_{j\neq \sigma(1)} (x_{\sigma(1)}-x_j)}\right)\left(\widetilde{U}_p\right) \in H^*(\Ps_1(r))$$ is the cohomology class defined in \eqref{q(U)def}.
\end{proposition}
\begin{proof}
Note that for any $q\in \C[\xb_1,...,\xb_r]^{\Sigma_r}$ the cohomology class $q(\widetilde{U}_p)$ belongs to $H^*(\Ms_1(r))$, hence to prove Proposition \ref{prop:intPtoM} it is enough to show that
\begin{multline}\label{eq:sym1}
\int_{\Ps_1(r)}f^*(\alpha)  \xb_1^k(\widetilde{U}_p) =  \int_{\Ms_1(r)}  \alpha\, \left(\sum_{\sigma\in \Sigma_r/\mathrm{Stab}_{1}} \frac{\sigma\cdot\xb_1^k}{\prod_{j\neq \sigma(1)} (x_{\sigma(1)}-x_j)}\right)\left(\widetilde{U}_p\right) = \\ 
  \int_{\Ms_1(r)} \alpha \, \left(\sum_{i=1}^r\frac{\xb_i^k}{\prod_{j\neq i} (\xb_{i}-\xb_j)}\right) \left(\widetilde{U}_p\right)
\end{multline}
for any $k\geq 0$.

We start by simplifying expression on the right-hand side of \eqref{eq:sym1}. 
\begin{lemma}\label{lem:H}
 Denote by $h_m(x_1,...,x_r)$ the complete homogeneous symmetric polynomial of degree $m$. Then assuming that $h_{0}(x_1,...,x_r) = 1$ and $h_{m}(x_1,...,x_r) = 0$ for $m<0$, we have 
 $$ h_{m-r+1}(x_1,...,x_r) = \sum_{i=1}^r \frac{x_i^m}{\prod_{j\neq i }(x_i-x_j)} \text{\;\; for any \;\;} m\geq 0.$$
\end{lemma}
\begin{proof}
Consider the generating function for the complete homogeneous symmetric polynomials:
\begin{equation}\label{eq:H1}
H(t) = \prod_{i=1}^r\frac{1}{1-x_it} = \sum_{k\geq 0}h_k(x_1,...,x_r)t^k.
\end{equation}
Note that the rational function $H(t)= \prod_{i=1}^r\frac{1}{(1-x_it)}$ can be written as 
$$H(t) = \sum_{i=1}^r \frac{C_i}{1-x_it},$$ where the constants $C_i$ are determined by the partial fraction decomposition:
$$C_i = \mathrm{lim}_{t\to1/x_i}H(t)(1-x_it) = \frac{1}{\prod_{j\neq i}(1-x_j/x_i)} = \frac{x_i^{r-1}}{\prod_{j\neq i}(x_i-x_j)}.$$
Substituting $C_i$ back into the partial fraction form of $H(t)$, we obtain that
\begin{equation}\label{eq:H2}
H(t) = \sum_{i=1}^r  \frac{x_i^{r-1}}{\prod_{j\neq i}(x_i-x_j)} \frac{1}{1-x_it}  = \sum_{k\geq 0}  \sum_{i=1}^r \frac{x_i^{r-1+k}}{\prod_{j\neq i}(x_i-x_j)}t^k.
\end{equation}
Now comparing the coefficients of $t^k$ in the expressions \eqref{eq:H1} and \eqref{eq:H2} for $H(t)$, we arrive at 
$$h_k(x_1,...,x_r) =  \sum_{i=1}^r \frac{x_i^{r-1+k}}{\prod_{j\neq i}(x_i-x_j)} \text{\;\; for any \;\;} k\in \mathbb{Z}.$$ 
Then substituting $ k=m-r+1$, we arrive at the statement of Lemma \ref{lem:H}. 
\end{proof}
Recall that in \S\ref{S:genP0}  we identified $\xb_1(\widetilde{U}_p)$ with the cohomology class $\xtil\in H^2(\Ps_1(r))$, defined in \eqref{eq:xxtil}. Thus it follows from  Lemma \ref{lem:H} that to prove \eqref{eq:sym1} it is enough to show the equality: 
 \begin{equation}\label{eq:sym2} 
\int_{\Ps_1(r)}f^*(\alpha)  \xtil^k(\widetilde{U}_p)  =   \int_{\Ms_1(r)} \alpha \, h_{k-r+1}(\widetilde{U}_p) 
\text{\;\; for\;\;} k\geq 0.
\end{equation}

Since $\Ps_1(r) \simeq \mathbb{P}(\widetilde{U}^*_p)$ (see  Remark \ref{rem:P1iso}) and $\xtil=c_1(\mathcal{O}_{\mathbb{P}(\widetilde{U}^*_p)}(1))$ (cf. \eqref{eq:xxtil}), equality  \eqref{eq:sym2}  follows from  \cite[Lemma 1]{MR623355}; this concludes the proof of Proposition \ref{prop:intPtoM}.
\end{proof}

Now using Proposition \ref{prop:intPtoM}, we rewrite \eqref{eq:intP01} as 
\begin{multline*}
 \int_{\Ms_1(r)}\exp(\delta_2\ftil_2+\delta_3\ftil_3+...+\delta_r\ftil_r) \prod_{k=2}^r\atil_k^{m_k} \prod_{j_k=1}^{2g}(\btil_k^{j_k})^{l_{k,j_k}}  \sum_{\sigma} \frac{\exp(-Q_{\sigma \check{s}}(\widetilde{U}_p)) (\sigma \xtil)^m}{\prod_{j\neq \sigma(1)} (x_{\sigma(1)}-x_j)},
 \end{multline*}
 where the sum runs over the elements $\sigma\in \Sigma_r/\mathrm{Stab}_{1}$.
According to Theorem \ref{ThmFKtoHam},  this integral is equal to 
 \begin{equation}\label{eq:int1}
\sum_{\bb\in\HH_r} 
\iber_{\mathbf{B},Q}\! \Biggl[\! \mathcal{T}(x) \sum_{i=1}^r \frac{\exp(Q_{\check{\alpha}^{i r}})\big(\frac{1}{r}\sum_{j=1}^r(x_i-x_j)\big)^m}{\prod_{j,  j\neq i}(x_i-x_j)}\!\Biggr](-[(1/r,...,1/r, 1/r-1)]_\bb),
\end{equation} 
where 
\begin{multline*}
\mathcal{T}(x) = \frac{(-1)^{{r \choose 2}(g-1)}}{r!} 
 \frac{ \prod_{k=2}^r\tau_k(x)^{m_k}}{\prod_{i<j}(x_i-x_j)^{2g-2}\, \mathrm{det}(\mathrm{Hess}_\bb(Q))}  \\
\int_{T^{2g}}\exp\left(-\sum_{a,b}\sum_{j=1}^g\zeta^j_a\zeta^{j+g}_b Q_{\check{u}_a\,\check{u}_b} \right) 
\prod_{k=2}^r\prod_{j_k=1}^{2g}\left(\sum_a \zeta_a^{j_k}\tau_{k_{\check{u}_a}}\right)^{l_{k,j_k}}.
\end{multline*}

Now using Remark \ref{rem:iBer2} with $v=x_i-x_r$, $i=1,2,...,r-1$, we rewrite \eqref{eq:int1} as
 \begin{equation}\label{eq:int2}
  \sum_{i=1}^r
 \sum_{\bb\in\HH_r} \iber_{\mathbf{B},Q} \Biggl[ \mathcal{T}(x) \frac{\big(\frac{1}{r}\sum_{j=1}^r(x_i-x_j)\big)^m}{\prod_{j,  j\neq i}(x_i-x_j)}\Biggr]\Big(-\Big[\sum_{j=1}^r\frac{x_j}{r}-x_i\Big]_\bb\Big).
\end{equation} 

Introducing the notation $\sigma_i$ for  the permutation $(i r)\in \Sigma_r$, we observe that 
 \begin{multline*}\label{eq:int3}
 \sum_{\bb\in\HH_r} \iber_{\mathbf{B},Q} \Biggl[ \mathcal{T}(x) \frac{\big(\frac{1}{r}\sum_{j=1}^r(x_i-x_j)\big)^m}{\prod_{j,  j\neq i}(x_i-x_j)}\Biggr]\Big(-\Big[\sum_{j=1}^r\frac{x_j}{r}-x_i\Big]_\bb\Big) {=} \\
  \sum_{\bb\in\HH_i} \iber_{\sigma_i\mathbf{B},Q} \Biggl[ \mathcal{T}(x) \frac{\sigma_i\cdot\big(\frac{1}{r}\sum_{j=1}^r(x_r-x_j)\big)^m}{\sigma_i\cdot \prod_{j,  j\neq r}(x_r-x_j)}\Biggr](-[\sigma_i\cdot c]_{\sigma_i\bb}) \overset{\ref{rem:iBer3}}{=} \\
   \sum_{\bb\in\HH_i} \iber_{\mathbf{B},Q} \Biggl[ \mathcal{T}(x) \frac{\big(\frac{1}{r}\sum_{j=1}^r(x_r-x_j)\big)^m}{ \prod_{j,  j\neq r}(x_r-x_j)}\Biggr](-[ c]_{\bb}),
\end{multline*} 
where for the last equality we used Remark \ref{rem:iBer3}. 
Finally, using the independence of the sum
$$ \sum_{\bb\in\HH_i} \iber_{\mathbf{B},Q} \Biggl[ \mathcal{T}(x) \frac{\big(\frac{1}{r}\sum_{j=1}^r(x_r-x_j)\big)^m}{ \prod_{j,  j\neq r}(x_r-x_j)}\Biggr](-[ c]_{\bb})$$
 on the choice of the Hamiltonian basis $\HH_i$ \cite{Szimrn} (cf. also \cite[Theorem 4.4]{OTASz}), we arrive at the expression \eqref{eq:ourint}, which concludes the proof of Theorem \ref{Thm:intP0}. 
\end{proof}

\section{Our main result and its relation to Kiem's work}\label{S:6}
In this final section, we present our main result: explicit formulas for the  Poincar\'e--Verdier  pairing on $IH^*(\Ms_0(r))$  (see Theorem \ref{Thm:mainresult}). We also compare these formulas, in the rank-2 case, with Kiem’s computation of the intersection pairing on $IH^*(\Ms_0(2))$.
\subsection{Ample cohomology classes}
As discussed in  \S\ref{S:CanEmbedding} and \S\ref{S:PVpairing}, in order to deduce explicit formulas for the  Poincar\'e--Verdier pairing on $IH^*(\Ms_0(r))$ from the intersection numbers on $H^*(\Ps_0(r))$, 
it is necessary to identify (relatively) ample cohomology classes on the moduli spaces $\Ms_0(r)$ and $\Ps_0(r)$. In the following proposition, we carry out this identification.

\begin{proposition}\label{prop:ample}
Let $f_2, \x\in H^2(\Ps_0(r))$ be the classes introduced in \S\ref{S:genP0}, and let $\pi:\Ps_0(r)\to \Ms_0(r)$ be the map defined in Proposition \ref{prop:bundle}.  There exist an ample line bundle   $\mathcal{L}_0$ on $\Ms_0(r)$ and a $\pi$-ample line bundle $\mathcal{L}$ on $\Ps_0(r)$, such that
$$ f_2 = c_1(\pi^*\mathcal{L}_0), \;\; -\x=c_1(\mathcal{L}),$$
and  $\mathrm{deg}(\mathcal{L}\big|_{\pi^{-1}(V)}) = 1$ for every stable $V\in \Ms_0(r)$.
\end{proposition}
\begin{proof} 
We begin by showing that the class $f_2\in H^2(\Ps_0(r))$ is a pullback of an ample class in $H^2(\Ms_0(r))$. 
Following \cite{MR1231966}, we consider the universal bundle $U$  over $\Ms_1(r)\times C$ normalized such that its first Chern class, when restricted to $\Ms_1(r) \times \{p\}$, is equal to  $r\ftil_2$. 
According to Proposition \ref{prop:bundle}, there is an isomorphism $$\Ps_1(r)\simeq \mathbf{P}(U\big|_{\Ms_1(r)\times\{p\}})\overset{f}{\to}\Ms_1(r),$$ and thus over $\Ps_1(r)$ we have the tautological short exact sequence:
\begin{equation}\label{SES:O(1)}
0\to Q\to f^*U\big|_{\Ms_1(r)\times\{p\}} \to \mathcal{O}(1) \to 0.
 \end{equation}
A line-by-line repetition of the argument in \cite[Lemma 2.1]{MR1231966}, which treats the case $\Ms_0(2)$, yields
$$
\mathcal{O}(1) = (\pi \circ \mathcal{H})^* \mathcal{L}_0,
$$
where $\mathcal{L}_0$ is the ample generator of $\mathrm{Pic}(\Ms_0(r))$ and 
$\mathcal{H}\colon \Ps_1(r) \to \Ps_0(r)$ is the isomorphism induced by the tautological Hecke correspondence (see Proposition \ref{prop:H:10}).

We observe that the cohomology class $-r \xtil$, defined in \eqref{eq:xxtil}, is equal to the first Chern class of $Q\otimes\mathcal{O}(-1)$. Tensoring the short exact sequence \eqref{SES:O(1)} by $\mathcal{O}(-1)$ and  taking determinants, we obtain $$rf^*(\ftil_2)-ry = -r\xtil \Leftrightarrow y =f^*(\ftil_2) + \xtil, \text{\;\; where \;\;} y=c_1(\mathcal{O}(1)) = c_1((\pi\circ\mathcal{H})^*\mathcal{L}_0).$$ Comparing this identity with equation \eqref{eq:Thm:genP0}, we deduce  that $ y = \mathcal{H}^*f_2$, and hence $f_2 = c_1(\pi^*\mathcal{L}_0)$.

Now  let $\mathcal{U}_0$ be a universal bundle over $\Ps_0(r)\times C$, and denote by $\F_1\subset \F_2=\mathcal{U}_{0|p} := \mathcal{U}_0\big|_{\Ps_0(r)\times\{p\}}$ the corresponding universal flag bundle.  
From \eqref{eq:tau2} we have $$2rf_2 = \mathrm{pr}_*c_2(\mathrm{End}(\mathcal{U}_0)),$$ where $\mathrm{pr}:\Ps_0(r)\times C \to \Ps_0(r)$ is the natural projection. Applying \cite[Lemma 2.8]{OTASz}, we obtain
$$f_2 =c_1( \mathrm{det}(\mathcal{U}_{0|p})^{1-g}\otimes \mathrm{det}(\mathrm{pr}_*\mathcal{U}_0)^{-1}).$$
Consequently,  by \eqref{eq:xxtil}, the class $f_2 - \x$ is the first Chern class of the line bundle
\begin{equation}\label{eq:defL1}
\mathrm{det}(\mathcal{U}_{0|p})^{1-g}\otimes \mathrm{det}(\mathrm{pr}_*\mathcal{U}_0)^{-1} \otimes \mathcal{L}, \text{ \: where \:} \mathcal{L}:=\mathrm{det}(\mathcal{U}_{0|p})^{1/r}\otimes \mathcal{F}_1^{-1}.
\end{equation}  As explained in \cite[\S2.2 and Remark 2.7]{OTASz}, this line bundle is ample on $\Ps_0(r)$. 

Since $f_2 = c_1(\pi^*(\mathcal{L}_0))$ with $\mathcal{L}_0$ ample on $\Ms_0(r)$, and $f_2 - \x$ is ample on $\Ps_0(r)$, it follows that $-\x=c_1(\mathcal{L})$ is relatively ample.

It remains to verify that $\mathcal{L}$ restricts to a line bundle of degree $1$ on each fiber $\pi^{-1}(V)$ over stable $V\in \Ms_0(r)$. As observed in Theorem \ref{thm:dtpi}(i), the fiber of the morphism $\pi:\Ps_0(r)\to \Ms_0(r)$ over a point corresponding to a stable vector bundle $V$ is $$\pi^{-1}(V) = \mathbb{P}(V_p) \simeq \mathbb{P}^{r-1},$$ equipped with the tautological line bundle $\mathcal{O}_{\mathbb{P}^{r-1}}(-1)$. The restriction of $\mathcal{F}_1$ to $\pi^{-1}(V)$ coincides with $\mathcal{O}_{\mathbb{P}^{r-1}}(-1)$, while the remaining factor in the definition \eqref{eq:defL1} of $\mathcal{L}$ restrict trivially.  This completes the proof of Proposition \ref{prop:ample}. 
\end{proof}

Recall the notation $i\colon IH^*(\Ms_0(r)) \hookrightarrow H^*(\Ps_0(r))$ for the inclusion obtained by identifying $IH^*(\Ms_0(r))$ with the zeroth perverse piece $P_0H^*(\Ps_0(r))$ (see Proposition \ref{prop:lowestperv}).
Proposition \ref{prop:ample} implies the following statement concerning the cup product of intersection cohomology classes. 
\begin{corollary} 
For any $\alpha,\beta \in IH^*(\Ms_0(r))$, the cup product
$ i(\alpha)\cup i(\beta) $
can be written in $H^*(\Ps_0(r))$ as a linear combination of cohomology
classes that do not involve $\x$. 
\end{corollary}
\begin{proof}
We observe that, by Proposition \ref{prop:ample}, no cohomology class in 
$P_0H^*(\Ps_0(r))$ can be divisible by $\x$. Indeed, multiplication by the 
relatively ample class $\x$ strictly increases perversity by the relative Hard Lefschetz Theorem \ref{Thm:HardLef}.
\end{proof}

\subsection{Our main result}
Now we can describe the intersection pairing in the intersection cohomology $IH^*(\Ms_0(r))$.

\begin{theorem}\label{Thm:mainresult}
Let $a_k, b_k^j, f_k \in H^*(\Ps_0(r))$ be cohomology classes defined in \S\ref{S:genP0}, and let  $A\colon H^*(\Ps_0(r))\to H^{*+2}(\Ps_0(r))$  denote the operator given by cupping with the class $f_2 $.
Set $n = \dim(\Ms_0(r))$.
\begin{enumerate}[label=(\roman*)]
\item For each $0 \leq d \leq 2n$, the degree-$d$ intersection cohomology of $\Ms_0(r)$ admits the following description:
\begin{equation*}
IH^d(\Ms_0(r)) = \bigoplus_{k\geq 1} \left( \mathrm{ker}(A^{n-d+k})\cap \mathrm{im}(A^{k-1})\right) \cap H^d(\Ps_0(r)).
\end{equation*}

\item   Let $i\colon IH^*(\Ms_0(r)) \hookrightarrow H^*(\Ps_0(r))$ denote the natural inclusion. 
Let $\alpha, \beta \in IH^*(\Ms_0(r))$ be classes satisfying $\mathrm{deg}\alpha+\mathrm{deg}\beta = 2\mathrm{dim}(\Ms_0(r))$, so that 
 $$ i(\alpha) \cup i(\beta)= \prod_{k=2}^r\left( a_k^{m_k} f_k^{n_k}\prod_{j_k=1}^{2g} (b_k^{j_k})^{l_{k,j_k}} \right).$$
Then the intersection Poincar\'e-Verdier pairing of $\alpha$ and $\beta$ in $IH^*(\Ms_0(r))$ is given by 
\begin{multline}\label{eq:final}
N \cdot\!
\mathrm{Coeff}_{\delta^{\mathbf{n}}}\Biggl[ \!\sum_{\bb\in \HH_n} \!
\iber_{\mathbf{B},Q}\! \Biggl[\!\frac{\!\big(\frac{1}{r}\sum_{j=1}^r\!(x_j-x_r)\big)^{r-1}  \prod_{k=2}^r\tau_k(x)^{m_k}}{\prod_{i=1}^{r-1}\!(x_r-x_i)\prod_{i<j}(x_i-x_j)^{2g-2}\, \mathrm{det}(\mathrm{Hess}_\bb(Q))}  \\
\int_{T^{2g}}\exp\left(-\sum_{a,b}\sum_{j=1}^g\! \zeta^j_a\zeta^{j+g}_b Q_{\check{u}_a\,\check{u}_b} \right) 
\prod_{k=2}^r\prod_{j_k=1}^{2g}\!\left(\!\sum_a \zeta_a^{j_k}\tau_{k_{\check{u}_a}}\!\right)^{l_{k,j_k}}\!\Biggr](-[c]_\bb)\!\Biggr], 
\end{multline}
where $$N=(-1)^{{r \choose 2}(g-1)} \frac{\prod_{k=2}^rn_k!}{(r-1)!} \text{\; and \;} c = (1/r,...,1/r,1/r-1).$$ 

As above, $Q(x)=\sum_{k=2}^r \delta_k \tau_k(x)$ is a polynomial on $V \otimes_{\mathbb{R}} \mathbb{C}$, where $\tau_k$ denotes the $k$-th elementary symmetric polynomial, $\delta_k$ is a formal nilpotent
parameter, and $\delta_2 \neq 0$.
We expand $\exp(Q_{\check{u}_a\,\check{u}_b})$ as a formal power series in the variables $\delta_k$, and we denote by $\mathrm{Coeff}_{\delta^{\mathbf{n}}}$ the coefficient of
$\delta^{\mathbf{n}} := \prod_{k=2}^r \delta_k^{n_k}.$
\end{enumerate}
\end{theorem}
\begin{proof}
Since $f_2\in H^*(\Ps_0(r))$ is the pullback of an ample class on $\Ms_0(r)$ (see Proposition \ref{prop:ample}), part (i) follows immediately from Definition \ref{def:pervfiltr}.

For part (ii), putting together Propositions \ref{prop:PDtoPV} and \ref{prop:ample}, we see that the intersection pairing of the classes $\alpha$ and $\beta$ in $IH^*(\Ms_0(r))$ is given by
\begin{equation*}
\int_{\Ps_0(r)}(-\x)^{r-1}\prod_{k=2}^r\left( a_k^{m_k} f_k^{n_k}\prod_{j_k=1}^{2g} (b_k^{j_k})^{l_{k,j_k}}\right), 
\end{equation*}
where $\x\in H^2(\Ps_0(r))$ is the class defined it \eqref{eq:xxtil}.
Applying Theorem \ref{Thm:intP0} to evaluate the integral over the moduli space $\Ps_0(r)$, we arrive at the expression \eqref{eq:final}.
\end{proof}

\begin{remark}
Remark \ref{rem:toJKKW1} allows us to compare our formula \eqref{eq:final} for the intersection pairing with the corresponding expression of Jeffrey--Kirwan--Kiem--Woolf \cite{JKKW06}. Under the identification described there, the factor
$(\frac{1}{r}\sum_{j=1}^r (x_j - x_r))^{r-1},$ which arises from the power of the relatively ample class, is replaced by
$\frac{1}{r}\prod_{j=1}^{r-1}(x_j - x_r),$ representing the top Chern class of the virtual relative tangent bundle. With this substitution, our formula recovers the form of the expression in \cite[Theorem 33]{JKKW06}.
\end{remark}

\subsection{Relation to Kiem's work}\label{S:Kiem}
We conclude the paper by comparing our formulas with those of Kiem \cite{Kiem}, who computed the intersection Poincar\'e pairing on the moduli space $\Ms_0(2)$ of rank-two semistable bundles.
As a first step, we write down the rank-two case of Theorem \ref{Thm:mainresult} in a form suitable for comparison.

 \begin{proposition}\label{prop:Rank2we}
 Let $a_2, f_2, \x, b_2^j \in H^*(\Ps_0(2))$, for $j=1,2,...,2g$, be the cohomology classes  defined in  \S\ref{S:genP0} for $r=2$, and set $\gamma = \sum_{j=1}^g\! b_2^jb_2^{j+g} \in H^6(\Ps_0(2))$.  Let $\alpha_1, \alpha_2 \in IH^*(\Ms_0(2))$ be classes so that 
  $$ i(\alpha_1) \cup i(\alpha_2)= a_2^{m} f_2^{n}\gamma^p \text{\;\; with \;\;} 2m+n+3p = 3g-3.$$
Then the  Poincar\'e--Verdier pairing of $\alpha$ and $\beta$ in $IH^*(\Ms_0(2))$ is given by 
\begin{equation*}
\frac{(-1)^{g+m}}{2^{1+2m+p-g}}\cdot\frac{n!g!}{(g-p)!}\cdot \underset{y=0}{\res} \frac{y^{2+2m+2p-2g}}{1-e^{-y}}\,dy.
\end{equation*}
\end{proposition}
\begin{proof}
We begin by applying Theorem \ref{Thm:mainresult} in the case $r=2$.   It follows that the intersection pairing of $\alpha_1$ and $\alpha_2$ is given by
\begin{multline*}
N \! \cdot
\mathrm{Coeff}_{\delta^{n}}\Biggl[\!
\iber_{\mathbf{B},Q}\! \Biggl[ \!\frac{\big((x_1-x_2)/2\big)  \tau_2(x)^{m}}{-(x_1-x_2)^{2g-1}\, \mathrm{det}(\mathrm{Hess}_\bb(Q))}  \\
\int_{T^{2g}}\!\exp\left(-\sum_{j=1}^g\! \zeta^j\zeta^{j+g}Q_{\check{u}\,\check{u}} \right) 
\!\left(\!\sum_{j=1}^g\!\zeta^{j}\zeta^{j+g}(\tau_{2_{\check{u}}})^2\!\right)^{p}\Biggr](-[c]_\bb)\!\Biggr],
\end{multline*}
where $N=(-1)^{g-1}(n)!$, $\bb = (\alpha^{12})$ (see \S\ref{S:NotationBases}), $Q(x) = -{\delta}(x_1-x_2)^2/4$ (see \eqref{eq:tau2}) $u= \frac{1}{\sqrt{2}}(x_1-x_2)$, and $\zeta^j\in H^1(T^{2g})$, for $j=1,...,2g$, are cohomology classes defined in \S\ref{S:cohTorus}. 
Using Remark \ref{rem:iBer1}, this expression can be rewritten as
\begin{equation}\label{eq:rk2-1}
 N  \cdot
\mathrm{Coeff}_{\delta^{n}}
\underset{y=0}{\res}\frac{\big(y/2\big)  (-y^2/4)^{m}}{-y^{2g-1}(1-e^{-\delta y})}
\!\int_{T^{2g}}\!\exp\left(\!\delta \sum_{j=1}^g\! \zeta^j\zeta^{j+g}\!\right) 
\!\left(\!\sum_{j=1}^g\frac{\zeta^j\zeta^{j+g}y^2}{2}\!\right)^{p}\!dy.
\end{equation}
Now set  $\xi = \sum_{j=1}^g\zeta^j\zeta^{j+g} \in H^2(T^{2g})$.  It is a simple exercise  to check (cf. \eqref{eq:intTr}) that
\begin{equation*}
\int_{T^{2g}}\!\exp(\delta \xi)\xi^p = \delta^{g-p}2^g\frac{g!}{(g-p)!}.
\end{equation*}
Substituting this into \eqref{eq:rk2-1}, we obtain
\begin{multline}\label{eq:rk2-2}
N  \cdot \mathrm{Coeff}_{\delta^{i+n}}\! \frac{(-1)^{m+1}g!}{(g-p)!}\underset{y=0}{\res}\frac{y^{2+2m+2p-2g}\delta^{g-p}}{2^{1+2m+p-g}(1-e^{-\delta y})} dy  = \\
N  \cdot \frac{(-1)^{m+1}g!}{(g-p)!}\underset{y=0}{\res}\frac{y^{2+2m+2p-2g}}{2^{1+2m+p-g}(1-e^{-y})} dy,
\end{multline}
where we used the identity $n -(g-p) = -1-(2-2g +2m+2p)$. 
\end{proof}

Recall from \S\ref{S:genP0} that we defined the cohomology classes $a_2, f_2, b_2^j\in H^*(\Ps_0(2))$ as the K\"unneth components of the second Chern class of the vector bundle $\overline{U}=\mathcal{U}\otimes \mathrm{det}(\mathcal{U})^{-1/2}$, where $\mathcal{U}\to \Ps_0(2)\times C$ is a universal bundle. Explicitly, we have
$$c_2(\overline{U}) = a_2\otimes 1 + \sum_{j=1}^{2g}b_2^j\otimes e_j+ f_2\otimes \omega\in H^*(\Ps_0(2))\otimes H^*(C),$$ where $\omega\in H^{2}(C)$ is the fundamental class of the curve and $\{e_{j}\}$ is a symplectic basis of $H^{1}(C)$. 
The key point underlying this definition is that the vector bundle $\overline{U}$, and hence the classes $a_2, f_2, b_2^j\in H^*(\Ps_0(2))$, are independent of the choice of the universal bundle $\mathcal{U}$. 

Motivated by Kiem's construction (see \cite[\S5]{Kiem}), we may alternatively consider the K\"unneth components of the bundle $\mathrm{End}(\mathcal{U})$, which is also independent on the choice of the universal bundle $\mathcal{U}$. In this case, we define $\alpha, \beta, \psi_j\in H^*(\Ps_0(2))$ as
$$c_2(\mathrm{End}(\mathcal{U})) = -\beta\otimes 1 + 4\sum_{j=1}^{2g}\psi_j\otimes e_j+ 2\alpha\otimes \omega\in H^*(\Ps_0(2))\otimes H^*(C).$$
Then we obtain the following statement, which agrees with Kiem's result; see \cite[Corollary 5.4]{Kiem}.
\begin{proposition}\label{prop:wetoKiem}
Let $\alpha,\beta\in H^*(\Ps_0(2))$ be the cohomology classes defined above. 
\begin{enumerate}[label=(\roman*)]
\item The intersection pairing of classes $\alpha_1, \alpha_2\in IH^*(\Ms_0(2))$ satisfying
$$i(\alpha_1)\cup i(\alpha_2)=  \beta^{m} \alpha^{n} \text{\;\; with \;\;} 2m+n = 3g-3 \text{\;\; and \;\;} m<g-1$$ is equal to
$(-1)^{g} n!\, 2^{2g-2} \kappa_{g-1-m}$, where $\kappa_{g-1-m}$ is defined by $\frac{t}{\tanh(t)}=\sum_{j\geq 0}\kappa_jt^{2j}.$ 
\item Moreover, the intersection-homology fundamental class of $\Ms_0(2)$ is represented by  $$\frac{\alpha^{g-2}\beta^{g-2}\psi}{(g-2)!(-4)^{g-1}}, \text{\; where \;}  \psi=\alpha\beta-4\sum_{j=1}^g{\psi_j\psi_{j+g}}.$$
\end{enumerate}
\end{proposition}
\begin{proof}
Using \eqref{eq:tau2}, we observe that  $c_2(\overline{U}) = \frac{1}{4}c_2(\mathrm{End}(\mathcal{U}))$, and therefore the corresponding K\"uneth components satisfy
$$\beta = -4 a_2, \;\; \alpha = 2f_2, \;\; \psi_j=b_2^j,\; j=1,2,...,2g.$$
Applying Proposition \ref{prop:Rank2we}, we obtain that the intersection pairing of classes $\alpha_1, \alpha_2\in IH^*(\Ms_0(2))$ is given by
\begin{multline}\label{eq:intweKiem}
\frac{(-1)^{g}\, n!\, 2^{2g-2}}{2^{2+2m-2g}}\cdot \underset{y=0}{\res}\frac{y^{2+2m-2g}}{1-e^{-y}}\,dy = \\
\frac{(-1)^{g}\, n!\, 2^{2g-2}}{2^{2+2m-2g}}\cdot \frac{1}{2}\underset{y=0}{\res}\,y^{2+2m-2g}\left(\frac{1}{1-e^{-y}}-\frac{1}{1-e^y}\right)dy = \\
\frac{(-1)^{g}\, n!\, 2^{2g-2}}{2^{2+2m-2g}}\cdot \frac{1}{2}\underset{y=0}{\res}\,y^{2+2m-2g}
\frac{e^{y/2}+e^{-y/2}}{e^{y/2}-e^{-y/2}}dy =\\
(-1)^{g}\, n!\, 2^{2g-2}\underset{y=0}{\res}\,y^{2+2m-2g}
\frac{e^{y}+e^{-y}}{e^{y}-e^{-y}}dy = (-1)^{g}\, n!\, 2^{2g-2}\kappa_{g-1-m}.
\end{multline} 
This verifies (i). To prove (ii), we first observe by a straightforward calculation that the class  $$\alpha^{g-2}\beta^{g-2}\psi\in H^{3g-3}(\Ps_0(2))$$ lies in the image of the natural inclusion $i\colon IH^*(\Ms_0(2))\hookrightarrow H^*(\Ps_0(2))$, defined in \ref{Thm:mainresult}. 
Applying Proposition \ref{prop:Rank2we}, we then compute the intersection pairing $$\langle i^{-1}(\alpha^{g-2}\beta^{g-2}\,\psi), 1\rangle_{IH(\Ms_0(2))} = (-4)^{g-1}(g-2)!;$$ this shows (ii).
\end{proof}

\bibliographystyle{unsrt}
\bibliography{IntPair.bib}

\end{document}